\documentclass[10pt]{amsart}
\usepackage{amsmath,amsthm,amssymb}
\usepackage{latexsym}

\def\rr{\mathbb R}

\usepackage{color}

\newtheorem{theoreme}{Theorem}[section]
\newtheorem{proposition}[theoreme]{Proposition}
\newtheorem{lem}[theoreme]{Lemma}
\newtheorem{corollary}[theoreme]{Corollary}

\newtheorem{rem}[theoreme]{Remark}

\newtheorem{definition}[theoreme]{Definition}

\theoremstyle{remark}

\numberwithin{equation}{section}

\begin{document}


\date{\today}

\title [Synchronization for Robin problem]{Approximate boundary synchronization by groups for a coupled system of wave equations with coupled  Robin boundary conditions}

\maketitle

\centerline {Tatsien Li\footnote {\rm   Corresponding author, School of Mathematical Sciences, Fudan University, Shanghai 200433, China;  Shanghai Key Laboratory for Contemporary Applied Mathematics;  Nonlinear Mathematical Modeling and Methods Laboratory, dqli@fudan.edu.cn;} \quad 
Bopeng  Rao\footnote {\rm
Institut de Recherche Math\'ematique Avanc\'ee, Universit\'e de Strasbourg, 67084  Strasbourg, France, bopeng.rao@math.unistra.fr }}

\bigskip

\begin{abstract} In this paper, we first give an algebraic characterization of  uniqueness of continuation  for a coupled system of wave equations with coupled  Robin boundary conditions. Then, the approximate boundary controllability and the approximate boundary synchronization by groups for a coupled system of wave equations with coupled  Robin boundary controls  are  developed around this fundamental characterization.\end{abstract}

\bigskip


\bigskip


\vskip1cm

\section{Introduction}

The phenomenon of synchronization was  observed by Huygens in 1665 \cite{Hugense}.
The first related mathematical research goes back to   Wiener \cite{wiener} in 1960's. The previous studies focused only on systems described by
ODE. The synchronization in the PDE case was first studied
for a coupled system of wave equations with Dirichlet boundary controls  by Li and Rao in  \cite{excD,  JMPA2016}  for the exact boundary synchronization, and  in \cite{RaoSicon, Cocv2017}  for the approximate  boundary synchronization.
Later, the synchronization for a coupled system of wave equations with Neumann boundary controls was carried  in \cite{ExConNeumann, ExSynNeumann}.  The most part  of the results was recently collected in the monograph  \cite{Book2019}.

In the framework of classical solutions,  the exact boundary  synchronization for a coupled system of 1-D wave equations with various boundary controls was   considered in  \cite{hulong, hulong2} for linear and  quasilinear cases.

Now let $\Omega\subset\mathbb R^n$ be a bounded domain  with smooth
boundary $\Gamma = \Gamma_1\cup\Gamma_0$ such that ${\overline
\Gamma_1}\cap{\overline \Gamma_0} = \emptyset.$
Let $A, B$ be  two matrices of order $N$ and $D$
a full column-rank matrix of order $N\times M$ with $M\leqslant N$.   Let
$U=(u^{(1)}, \cdots, u^{(N)})^T \hbox{ and } H=(h^{(1)}, \cdots, h^{(M)})^T$
stand for  the state variables and the boundary controls applied  on $\Gamma_1$, respectively.
Consider  the following coupled
system of wave equations with coupled Robin  boundary controls:
\begin{equation}\label {0.1}\left\lbrace
\begin{array}{ll} U''-\Delta U+ AU=0& \hbox{in }  (0, +\infty)\times \Omega, \\
U= 0& \hbox{on } (0, +\infty)\times \Gamma_0,\\
\partial_\nu U+BU =  DH& \hbox{on }   (0, +\infty)\times\Gamma_1\end{array}
\right.
\end{equation}
with  the initial condition
\begin{equation}\label{0.2}t=0:\quad U=\widehat U_0, \quad
U'=\widehat U_1 \quad  \hbox{in } \Omega,\end{equation}
where $\partial_\nu$ denotes the outward normal derivative.

In \cite{RobinExact}, the exact boundary  controllability for system  \eqref{0.1} was established as $M=N$. In the case of fewer boundary controls, namely,  as $M<N,$ the non-exact boundary  controllability was also established  for a parallelepiped domain. Moreover, the exact boundary synchronization  by $p$-groups was also studied under the condition $M=N-p$.
The exact boundary controllability  as well as  the exact boundary synchronization by $p$-groups are intrinsically linked with the number of applied boundary controls.
In order to reduce the number of boundary controls, we return to consider the  approximate boundary controllability and the approximate boundary synchronization by $p$-groups.

Consider the following  system for the adjoint variable $\Phi= (\phi^{(1)},\cdots, \phi^{(N)})^T:$
\begin{equation}\label{0.5}
\begin{cases} \Phi''-{\Delta} \Phi+A^T\Phi=0&\hbox{in }   (0, +\infty) \times \Omega,\\
\Phi = 0&\hbox{on }   (0, +\infty) \times \Gamma_0,\cr
\partial_\nu \Phi + B^T\Phi= 0 &\hbox{on } (0, +\infty) \times \Gamma_1 \end{cases}
\end{equation}
with the initial data
\begin{equation}\label{0.6} t=0:\quad \Phi=\widehat\Phi_0, \quad  \Phi'=\widehat\Phi_1\quad  \hbox{in }  \Omega.\end{equation}

We say (see Definitions \ref{th3.2} and \ref{th4.2}  below) that system  (\ref{0.1}) is approximately
controllable  at the time $T>0$,  if for any given  initial data $(\widehat U_0, \widehat U_1)$, there exists a  sequence
$\{H_n\}$ of boundary controls, such that the corresponding
sequence $\{U_n\}$ of solutions  goes to zero for $t\geqslant T$
as $n\rightarrow +\infty$. Accordingly  the adjoint system \eqref{0.5} is $D$-observable on  a finite time interval $[0,T]$, if the $D$-observation
 \begin{align}\label{0.7}
D^T \Phi \equiv 0 \quad  \hbox{on  } [0,T] \times \Gamma_1\end{align}
implies that  $\Phi \equiv 0.$

Similar to Dirichlet boundary controls in \cite{AppSynDirichlet}, the approximate boundary controllability  of system  \eqref{0.1}  is still equivalent to the $D$-observability of the adjoint system \eqref{0.5}.
The main interest of the approximate boundary controllability consists in the fact that  the rank $M$ of the matrix $D$ for realizing it  may  be   substantially smaller than the number $N$ of state variables.

For Dirichlet boundary controls,  it  was shown in \cite{RaoSicon} that the following Kalman's criterion
\begin{equation}\hbox{rank} (D, AD,\cdots, A^{N-1}D) =  N\end{equation} is  necessary for the $D$-observability of the corresponding adjoint system.
For coupled Robin boundary controls, we want to find a similar characterization  on the  matrices $A, B$ and $D$, which is necessary for
the $D$-observability
of the  adjoint system \eqref{0.5}.  But the situation  seems to be more complicated because of the presence of the second coupling matrix $B$.

Let $V$ be a subspace, which  is contained in Ker$(D^T)$ and invariant for both $A^T$ and $ B^T$. We observe that  system \eqref{0.5} is not $D$-observable in the subspace $V$.

We will construct a composite matrix $\mathcal R$ (see \eqref{2.2} below)  to characterize the subspace $V$  of this kind.
We will show that $\hbox{Ker}(\mathcal R^T)$ is the largest subspace of all the subspaces  which are  contained in Ker$(D^T)$ and invariant for $A^T$ and $ B^T$ (see Lemma \ref{th2.1} below).
As a direct consequence,  $\hbox{Ker}(\mathcal R^T)=\{0\}$ is a necessary condition   for the $D$-observability of the adjoint system \eqref{0.5}. The approximate boundary controllability will be first developed around this fundamental characterization.

Next, when $\hbox{Ker}(\mathcal R^T)= \mbox{Span}\{E_1, \cdots, E_p\}$, assume  that both $A$ and $B$ admit a common invariant subspace $\mbox{Span}\{e_1, \cdots, e_p\}$ such that
$(E_i, e_j)=\delta_{ij} (i, j=1, \cdots, p)$, then
both $\mbox{Span}\{e_1, \cdots, e_p\}$ and $ \mbox{Span}\{E_1, \cdots, E_p\}^\perp$  are invariant for $A$ and $B$. Moreover, the projection of system \eqref{0.1} on $\mbox{Span}\{e_1, \cdots, e_p\}$ is independent  of the applied boundary controls, therefore uncontrollable, while,  the projection  of system \eqref{0.1} on $\mbox{Span}\{E_1, \cdots, E_p\}^\perp$ is approximately null controllable. This is the basic idea  that  we develop in this  paper for the approximate boundary synchronization by $p$-groups.

The paper is organized as follows. In \S 2, we give an algebraic Lemma,  which  generalizes Kalman's criterion or Hautus test.   In  \S3, we establish the well-posedness of problems.  \S4 is  devoted
to the $D$-observability and the  approximate boundary null controllability.   Generally speaking,  the condition $\hbox{dim Ker}(\mathcal R^T)=0$  is not sufficient for   the uniqueness of continuation  of solutions to  the adjoint  system \eqref{0.5} with $D$-observation \eqref{0.6}. In fact, this is not  a standard type of Holmgren's uniqueness  theorem.
In \S5, we outline some known results on the topic. In  \S 6, we consider the  approximate boundary synchronization of system \eqref{0.1} in the case that $\hbox{dim Ker}(\mathcal R^T)\not =0$. For this purpose,  we first show that $\hbox{dim Ker}(\mathcal R^T)\leqslant 1$ is a  necessary condition for the  approximate boundary synchronization (Theorem \ref{th6.1}). Then, under the condition that $\hbox{dim Ker}(\mathcal R^T) = 1$, we show the necessity of the condition of $C_1$-compatibility for  the coupling matrices $A$ and $B$ related  to the synchronization matrix $C_1$, the independence  of the approximately  synchronizable states  with respect to the applied boundary controls, and the  approximate boundary synchronization under  the condition of $C_1$-compatibility (Theorems \ref{th6.2} and \ref{th6.3}). In \S7, we generalize the above consideration to the  approximate boundary synchronization by $p$-groups and carry on the study from  a
general point of view. In \S8, based on the sharp regularity  in Lasiecka and Triggiani \cite{Lasciecka1, Lasciecka2} on the solution  to  the wave equation with Neumann boundary conditions, we   establish the necessity  of some algebraic properties on the matrices $A$ and $ B$ for the existence of the approximately synchronizable state.

Let us comment some  related literatures. One of the motivation of studying  the synchronization  consists  of establishing a weak exact boundary controllability  in the case of  fewer boundary controls.
In order to realize the exact boundary controllability, because of its uniform character with respect to the state variables, the number of  boundary controls must be equal to  the degrees of freedom of the considered system. However,
when   the components of initial data are allowed to  have different levels of energy,
 the exact boundary  controllability  by means of only one  boundary control for a system of two
wave equations   was established  in  Liu and Rao \cite{liu-rao}, Rosier and de Teresa \cite{rosier}, and
for a cascade system of $N$ wave equations  in   Alabau-Boussouira   \cite{alabaub}.  In \cite{Dehman}, Dehman {\it et al}  established
the controllability of two coupled wave equations on a compact
manifold with only one local distributed control. Moreover, both the optimal time
of controllability  and the controllable spaces are given in the
cases with  the same or
different wave speeds.

The approximate boundary null controllability  is more flexible with respect to the number of applied boundary controls.
In Li and Rao \cite{RaoSicon} as well as in the present paper, for a coupled system of wave equations with Dirichlet/Neuman/Robin boundary controls, some fundamental  algebraic properties on the coupling matrices are  used to characterize the uniqueness of continuation for the solution to the corresponding adjoint systems.
Although these  criteria are only necessary in general, they open  an important way to the research on the uniqueness of continuation for the system of hyperbolic partial differential equations.

In contrast with hyperbolic systems, in Ammar Khodja \cite{Dehman} (also \cite{ammar2} and the reference therein), it was shown that  Kalman's criterion is sufficient to the exact boundary  null  controllability for systems of parabolic  equations.
Recently,  Wang and Yuan \cite{wanglijuan} have established the minimal time for a control problem related to the exact synchronization for a linear parabolic system.

The average controllability proposed by Zuazua in \cite{zuazua, luq} gives another way to deal with the controllability with fewer controls. The observability inequality is particularly interesting for a trial on the decay rate of  approximate  controllability.

\vskip1cm

\section{An algebraic Lemma}
Let $A $ be  a  matrix of order $N$ and $D$
a full column-rank matrix of order $N\times M$ with $M\leqslant N$.
We have shown that the following Kalman's criterion (see \cite{RaoSicon}):
\begin{equation}\hbox{rank} (D, AD,\cdots, A^{N-1}D) \geqslant   N-d\label{con}\end{equation}holds if and only if the dimension of any given
subspace,  contained in $\hbox{Ker}(D^T)$ and invariant for  $A^T$,
does not exceed $d$.
In particular,  the equality holds  if and only if the dimension of the largest subspace,   contained in $\hbox{Ker}(D^T)$ and  invariant for  $A^T$, is exactly equal to $d$.

Let $A, B$ be  two matrices of order $N$ and $D$
a full column-rank matrix of order $N\times M$ with $M\leqslant N$.
For any given non-negative  integers   $p, q,\cdots, r,s\geqslant 0$, we define a  matrix
of order $N\times M$  by
\begin{equation} \mathcal R_{(p, q ,\cdots,  r, s)} = A^pB^q\cdots A^rB^sD.\label{2.1}\end{equation}
We  construct  an enlarged  matrix
\begin{equation} \mathcal R= (\mathcal R_{(p, q ,\cdots,  r, s)},  \mathcal R_{(p', q' ,\cdots,  r', s')},\cdots )\label{2.2}\end{equation}
by   the matrices $\mathcal R_{(p, q, \cdots, r,s)}$ for all possible  $(p, q,\cdots, r,s)$, which, by   Theorem of Caylay-Hamilton, essentially constitute  a finite set $\mathcal M$ with dim$(\mathcal M) \leqslant MN$.

\begin{lem}  \label{th2.1}  $\hbox{Ker}(\mathcal R^T)$ is the largest subspace of all the subspaces  which are  contained in Ker$(D^T)$ and invariant for $A^T$ and $ B^T$. \end{lem}

\begin{proof}  First, noting that  $\hbox{Im}(D)\subseteq \hbox {Im}(\mathcal R)$, we have $\hbox{Ker}(\mathcal R^T)\subseteq \hbox {Ker}(\mathcal D^T)$. We now show that   $\hbox{Ker}(\mathcal R^T)$   is  invariant for  $A^T$ and $ B^T$.   Let $x\in \hbox{Ker}(\mathcal R^T)$. We have
\begin{equation} D^T(B^T)^s(A^T)^r\cdots (B^T)^q(A^T)^p x=0\end{equation}
for any given integers $p,  q, \cdots, r, s\geqslant 0.$
Then, it follows that  $A^Tx\in  \hbox{Ker}(\mathcal R^T)$, namely, $\hbox{Ker}(\mathcal R^T)$ is invariant for $A^T.$ Similarly,  $\hbox{Ker}(\mathcal R^T)$ is invariant for $B^T.$ Thus,  the subspace
 $ \hbox{Ker}(\mathcal R^T)$ is contained in $\hbox{Ker}(D^T)$ and  invariant for both  $A^T$ and $ B^T$.

Now let $V$ be another subspace,  contained in Ker$(D^T)$ and invariant for  $A^T$ and $ B^T$. For any given $y\in V$, we have
\begin{equation}D^Ty=0,\quad  A^Ty\in V, \quad B^Ty\in V.\label{2.2b}\end{equation}
Then, it is easy to see that
\begin{equation} (B^T)^s(A^T)^r\cdots (B^T)^q(A^T)^py\in V\end{equation}
for any given integers $p,  q, \cdots, r, s\geqslant 0.$
Thus, by the first formula of \eqref{2.2b}  we have
\begin{equation} D^T(B^T)^s(A^T)^r\cdots (B^T)^q(A^T)^p y =0 \end{equation}
 for any given integers $p, q,\cdots, r,s\geqslant 0$,  namely, we have
\begin{equation} V\subseteq   \hbox{Ker}(\mathcal R^T).\label{2.3}\end{equation}
The proof is then complete.
\end{proof}

By the rank-nullity theorem, we have $\hbox{rank} (\mathcal R)   +  \hbox{dim Ker} (\mathcal R^T) = N.$
The following lemma is a dual version of Lemma \ref{th2.1}.

\begin{lem} \label{th2.2}  Let   $d\geqslant 0$ be an   integer.  Then

(i) the rank condition
\begin{equation}\hbox{rank} (\mathcal R) \geqslant N-d\label{2.4}\end{equation}
holds true  if and only if    the dimension of  any given subspace,  contained in Ker$(D^T)$ and invariant for  $A^T$ and $ B^T$,  does not exceed $d$;

(ii) the rank condition
\begin{equation}\hbox{rank} (\mathcal R) = N-d\label{2.4eq}\end{equation}
holds true if and only if   the dimension of  the largest  subspace,  contained in Ker$(D^T)$ and invariant for  $A^T$ and $ B^T$,  is exactly equal to $d$.
\end{lem}

\begin{proof} (i) Let $V$  be  a   subspace  which is  contained in $\hbox{Ker}(D^T)$ and invariant for  $A^T$ and $ B^T$. By Lemma \ref{th2.1}, we have
 \begin{equation}  V\subseteq Ker (\mathcal R^T).\label{2.4bb}\end{equation}

Assume that \eqref{2.4} holds, it follows from  \eqref{2.4bb} that
\begin{equation}N-d\leqslant\hbox{rank} (\mathcal R) = N - \hbox{dim Ker} (\mathcal R^T)   \leqslant  N - \hbox{dim}(V),\end{equation}
namely,
\begin{equation}\hbox{dim}(V)\leqslant d.\label{2.4aa}\end{equation}

Conversely, assume that \eqref{2.4aa} holds for   any given subspace  $V$ which is  contained in $\hbox{Ker}(D^T)$ and invariant for  $A^T$ and $ B^T$.  In particular, by Lemma \ref{th2.1}, we have $\hbox{dim  Ker}(\mathcal R^T)\leqslant d.$ Then it follows that
\begin{equation} \hbox{rank} (\mathcal R) = N - \hbox{dim Ker} (\mathcal R^T)   \geqslant  N - d,\end{equation}The proof is then complete.

(ii)  Noting that (\ref{2.4eq}) can be written  as
 \begin{equation}\hbox{rank} (\mathcal R)\geqslant  N-d\label{2.9cc}\end{equation}
 and
 \begin{equation}\hbox{rank} (\mathcal R) \leqslant  N-d.\label{2.10cc}\end{equation}
By (i), the rank condition (\ref{2.9cc}) means that dim$(V)\leqslant d$ for any given  invariant
subspace  $V$ of $A^T$ and $B^T$,  contained in $\hbox{Ker}(D^T)$. We claim that there exists a subspace $V_0$, which is contained in $\hbox{Ker}(D^T)$ and  invariant
for  $A^T$ and   $B^T$, such that dim$(V_0) = d$.   Otherwise, all the subspaces of this kind have dimension less than or equal to $(d-1)$. By (i),  we  get \begin{equation}\hbox{rank} (\mathcal R)\geqslant  N-d+1,\end{equation}
 which contradicts \eqref{2.10cc}. It proves (ii).
\end{proof}

\begin{rem}In  the special  case that $B=I$,  it is easy to see that
\begin{equation}\mathcal R =  (D, AD,\cdots, A^{N-1}D). \end{equation}
Then,  by Lemma \ref{th2.2}, we find again  (see \cite{RaoSicon}) that  Kalman's criterion \eqref{con}
holds if and only if the dimension of any given
subspace,  contained in $\hbox{Ker}(D^T)$ and invariant for  $A^T$,
does not exceed $d$.
In particular,  the equality holds  if and only if the dimension of the largest subspace,   contained in $\hbox{Ker}(D^T)$  and invariant for  $A^T$, is exactly equal to $d$.\end{rem}

\vskip 1cm

\section{Well-posedness }

Let $\Omega\subset\mathbb R^n$ be a bounded domain  with smooth
boundary $\Gamma = \Gamma_1\cup\Gamma_0$ such that ${\overline
\Gamma_1}\cap{\overline \Gamma_0} = \emptyset.$
Let
$$U=(u^{(1)}, \cdots, u^{(N)})^T \hbox{ and } H=(h^{(1)}, \cdots, h^{(M)})^T$$
stands for  the state variables and the boundary controls applied  on $\Gamma_1$, respectively.
Consider  the following coupled
system of wave equations with coupled Robin  boundary controls:
\begin{equation}\label {4.1}\left\lbrace
\begin{array}{ll} U''-\Delta U+ AU=0& \hbox{in }  (0, +\infty)\times \Omega, \\
U= 0& \hbox{on } (0, +\infty)\times \Gamma_0,\\
\partial_\nu U+BU =  DH& \hbox{on }   (0, +\infty)\times\Gamma_1\end{array}
\right.
\end{equation}
with the initial condition
\begin{equation}\label{4.2}t=0:\quad U=\widehat U_0, \
U'=\widehat U_1 \quad  \hbox{in } \Omega,\end{equation}
where $\partial_\nu$ denotes the outward normal derivative.

Accordingly, let$$\Phi= (\phi^{(1)},\cdots, \phi^{(N)})^T.$$
Consider the following  adjoint system
\begin{equation}\label{3.1}
\begin{cases} \Phi''-{\Delta} \Phi+A^T\Phi=0&\hbox{in }   (0, +\infty) \times \Omega,\\
\Phi = 0&\hbox{on }   (0, +\infty) \times \Gamma_0,\cr
\partial_\nu \Phi + B^T\Phi= 0 &\hbox{on } (0, +\infty) \times \Gamma_1 \end{cases}
\end{equation}
with the initial data
\begin{equation}\label{3.2} t=0:\quad \Phi=\Phi_0, \  \Phi'=\Phi_1\quad  \hbox{in }  \Omega.\end{equation}

Denote
$${\mathcal H}_0= L^2(\Omega),  \quad {\mathcal H}_1 =    H^1_{\Gamma_0}(\Omega),  \quad {\mathcal L} =   L_{loc}^2(0,+\infty; {L}^2(\Gamma_1))
$$
and by $\mathcal H_{-1}$ the dual space of $\mathcal H_1$ with respect to the pivot space $\mathcal H_0$, here $ H^1_{\Gamma_0}(\Omega)$ denotes the subspace of  $H^1(\Omega)$,  composed of functions with null trace on the boundary $\Gamma_0$.

We first consider  the adjoint system \eqref{3.1}   with the homogeneous boundary conditions by  a direct  method given in \cite{Lions2}, which   has the advantage of applying  the semi-group approach in \cite{pazy} in the present situation.

\begin{proposition} \label{th3.1}  Assume that the matrix $B$ is symmetric. Then for any given initial data $(\Phi_0,\Phi_1)\in  (\mathcal H_1)^N\times (\mathcal  H_0)^N$,  the adjoint problem (\ref{3.1})-(\ref{3.2}) admits a unique solution:
\begin{equation}\Phi\in C_{loc}^0([0, +\infty); (\mathcal H_1)^N)\cap C_{loc}^1([0, +\infty); (\mathcal H_0)^N).\label{3.3}\end{equation}\end{proposition}
\begin{proof}  We first  formulate  system (\ref{3.1}) into the following variational form:
\begin{equation}\int_\Omega (\Phi'',  \widehat \Phi)dx + \int_\Omega \langle \nabla\Phi, \nabla \widehat \Phi\rangle dx+
\int_{\Gamma_1}(B\Phi, \widehat  \Phi)d\Gamma +\int_\Omega (A\Phi,   \widehat  \Phi)dx =0 \label{3.4}\end{equation}
for any  given test function $ \widehat   \Phi \in (\mathcal H_1)^N$, where $(\cdot, \cdot)$ denotes the inner product of  $\mathbb R^N$, while
$\langle\cdot, \cdot\rangle$  denotes  the  inner product of $\mathbb M^{N\times N}(\mathbb R)$.
Recalling  the following  interpolation inequality
$$\int_\Gamma|\phi|^2d\Gamma\leqslant  c \|\phi\|_{H^1(\Omega)}\|\phi\|_{L^2(\Omega)},\quad \forall \phi\in H^1(\Omega),$$
 we have
$$\int_{\Gamma_1}(B\Phi, \Phi)d\Gamma \leqslant \|B\|\int_{\Gamma_1}|\Phi|^2d\Gamma
\leqslant c\|B\|\|\Phi\|_{(\mathcal H_1)^N}\|\Phi\|_{(\mathcal H_0)^N},$$
then it follows that
$$\int_\Omega \langle \nabla\Phi, \nabla\Phi\rangle dx+
\int_{\Gamma_1}(B\Phi,\Phi)d\Gamma +\lambda \|\Phi\|_{(\mathcal H_0)^N}^2\geqslant c'\|\Phi\|_{(\mathcal H_1)^N}^2$$
for some suitable constants $\lambda>0$ and $c'>0$. Therefore, the symmetric  bilinear form
$$ \int_\Omega \langle \nabla\Phi, \nabla \widehat  \Phi\rangle dx+
\int_{\Gamma_1}(B\Phi, \widehat   \Phi)d\Gamma $$ is coercive. Moreover, the non-symmetric part in \eqref{3.4} satisfies
$$\int_\Omega (A\Phi,   \widehat  \Phi)dx\leqslant \|A\|\|\Phi\|_{(\mathcal H_0)^N}\| \widehat  \Phi\|_{(\mathcal H_0)^N}.$$
By Theorem 1.1 (p. 151
in \cite{Lions2}),
the variational problem (\ref{3.4})  with the initial data (\ref{3.2}) admits a unique solution $\Phi$ with  (\ref{3.3}). The proof is complete.
\end{proof}
Now we  consider  problem \eqref{4.1}-\eqref{4.2} with inhomogeneous boundary conditions  by the duality method given in \cite{Lions}.
\begin{definition}  $U$ is a weak solution to   problem (\ref{4.1})-(\ref{4.2}), if
\begin{equation}\label{4.3}U\in C_{loc}^0([0, +\infty);  (\mathcal H_0)^N)\cap C_{loc}^1([0, +\infty);  (\mathcal H_{-1})^N)\end{equation}
such that
\begin{equation}\label{4.4}\begin{cases}
&\langle(U'(t),  -U(t)),   (\Phi(t), \Phi'(t))\rangle
= \langle(U_1,   -U_0), (\Phi_0, \Phi_1)\rangle
 \\& \displaystyle+\int_0^t \int_{\Gamma_1}(DH(\tau), \Phi(\tau) ) d\Gamma d\tau, \quad \forall t\geqslant 0\end{cases}\end{equation}
 holds for  the solution $\Phi$ to problem \eqref{3.1}
with any given initial data 
 $(\Phi_0, \Phi_1)\in ({\mathcal H}_1)^N \times ({\mathcal H}_0)^N$, here and hereafter
 $\langle\cdot,   \cdot\rangle$ denotes the dual product between $({\mathcal H}_{-1})^N \times ({\mathcal H}_0)^N$ and
 $({\mathcal H}_1)^N \times ({\mathcal H}_0)^N$.
\end{definition}
\begin{proposition}\label{th5.1} Assume that the matrix $B$ is symmetric.
Then for any given initial  data $(\widehat U_0, \widehat  U_1)\in ({\mathcal H}_0)^N \times ({\mathcal H}_{-1})^N$ and for any given boundary function  $H \in  {\mathcal L}^M$ with compact support in $[0, T]$,  problem (\ref{4.1})-(\ref{4.2})  admits a unique weak solution $U$.
Moreover,  the  linear mapping
\begin{equation}(\widehat U_0, \widehat U_1,H) \rightarrow (U, U')\end{equation}is continuous with respect to the corresponding topologies.
\end{proposition}
\begin{proof}  Define the linear form
$$L_t(\Phi_0,\Phi_1)  =\langle(\widehat U_1,  -\widehat U_0),   (\Phi_0, \Phi_1)\rangle+\int_0^t \int_{\Gamma_1}(DH(\tau), \Phi(\tau) ) d\Gamma d\tau.$$
Clearly, $L_t$ is bounded  in $(\mathcal H_1)^N\times (\mathcal H_0)^N$.  Let $S_t$ be the semi-group associated to the problem (\ref{4.1})-(\ref{4.2}) with the homogeneous boundary conditions  on the Hilbert space  $(\mathcal H_1)^N\times (\mathcal H_0)^N$.
The composed linear form  $L_t\circ S_t^{-1}$ is bounded  in $(\mathcal H_1)^N\times (\mathcal H_0)^N$. Then,  by Riesz-Fr\^echet's representation theorem, there exists a unique element
 $(U'(t), -U(t))\in  (\mathcal H_{-1})^N\times (\mathcal H_0)^N$, such that$$L_t\circ S_t^{-1}(\Phi(t), \Phi'(t)) =\langle  (U'(t), -U(t)), (\Phi(t),  \Phi'(t))\rangle$$
for any given  $(\Phi_0, \Phi_1)\in (\mathcal H_1)^N\times (\mathcal H_0)^N.$  Noting
$$L_t\circ S_t^{-1}(\Phi(t), \Phi'(t)) = L_t(\Phi_0, \Phi_1),$$
 we get  (\ref{4.4})  for any given  $(\Phi_0, \Phi_1)\in (\mathcal H_1)^N\times (\mathcal H_0)^N.$   Moreover, for any given $T>0$, we have
\begin{align}&\sup_{0\leqslant  t\leqslant T}\|(U'(t), -U(t))\|_{(\mathcal H_{-1})^N\times (\mathcal H_0)^N}\notag\\
\leqslant c_T\big(\|(\widehat U_1, &\widehat U_0)\|_{(\mathcal H_{-1})^N\times (\mathcal H_0)^N}+ \|H\|_{\mathcal L^M}\big),\notag\end{align}
where $c_T>0$ is a positive constant depending on $T$. This gives the continuous dependence.

Finally, by a classic argument of density, we get  the regularity (\ref{4.3}) for all initial data $(\widehat  U_0,  \widehat U_1)\in ({\mathcal H}_0)^N \times ({\mathcal H}_{-1})^N$£¬. The proof is then complete.
\end{proof}

\begin{rem}\label{th4.2aa} Suppose that $B$ is similar to a symmetric matrix.
Let $P$ be an invertible matrix such that $PBP^{-1}$ is symmetric.  The  new variable $\tilde U= PU$  satisfies   the same system  (\ref{4.1}) with the coupling  matrix $\tilde A= PAP^{-1}$ and the  symmetric  matrix $\tilde B =PBP^{-1}$.
Hence, in order to guarantee the well-posedness of problem (\ref{4.1})-(\ref{4.2}),   in what follows, we always assume that $B$ is similar to a symmetric matrix.
 \end{rem}

\vskip 1cm

\section{Approximate boundary null controllability}

\begin{definition}\label{th3.2}
For $(\Phi_0, \Phi_1)\in ({\mathcal H}_1)^N  \times  ({\mathcal H}_0)^N $, the adjoint system \eqref{3.1} is $D$-observable on  a finite interval $[0,T]$, if the observation
 \begin{align}\label{3.5}
D^T \Phi \equiv 0 \quad  \hbox{on }   [0,T] \times \Gamma_1\end{align}
implies that   $\Phi_0=\Phi_1\equiv 0$, then $\Phi \equiv 0.$
\end{definition}

\begin{proposition}\label{th3.3} If  the adjoint system \eqref{3.1} is $D$-observable, then  we necessarily have $\hbox{rank}(\mathcal R)=N$.
Conversely, if rank$(D)= N$, then system \eqref{3.1} is $D$-observable.
\end{proposition}
\begin{proof} Otherwise,  $\hbox{dim  Ker}(\mathcal R^T)=d\geqslant 1.$ Let Ker$(\mathcal R^T)= \hbox{Span}\{E_1,\cdots, E_d\}$. By Lemma \ref{th2.1},  $\hbox{Ker}(\mathcal R^T)$ is contained in  $\hbox{Ker}(D^T)$  and invariant for $A^T$ and $B^T$,
namely,   we have
\begin{equation}D^TE_r=0, \quad 1\leqslant r\leqslant d\label{3.6a}\end{equation}
and there exist  coefficients $\alpha_{rs}$ and $\beta_{rs}$   such that
\begin{equation}A^TE_r=\sum_{s=1}^d\alpha_{rs}E_s,\quad B^TE_r=\sum_{s=1}^d\beta_{rs}E_s,\quad 1\leqslant r\leqslant d.\label{3.6}\end{equation}
In what follows, we  restrict   system  \eqref{3.1} on the subspace $\hbox{Ker}(\mathcal R^T)$ and look for a solution of the form
\begin{equation}\Phi = \sum_{r=1}^d\phi_rE_r, \label{3.6b} \end{equation}
which, because of \eqref{3.6a}, obviously satisfies the $D$-observation \eqref{3.5}.

Inserting the function \eqref{3.6b} into   system \eqref{3.1}   and noting \eqref{3.6}, it is easy to see that
 for  $1\leqslant s\leqslant d$, we have
\begin{equation}\begin{cases}\phi_s''-\Delta \phi_s+ \sum_{r=1}^d\alpha_{rs}\phi_r=0 &\hbox{ in }   (0, +\infty)\times \Omega,\\
\phi_s=0  &\hbox{ on } (0, +\infty)\times \Gamma_0,\\
\partial_\nu\phi_s+\sum_{r=1}^d\beta_{rs}\phi_r=0 &
\hbox{ on }  (0, +\infty)\times \Gamma_1.\end{cases}\label{3.7}\end{equation}
For any non-trivial initial data:
\begin{equation}t=0:\quad \phi_s = \phi_{0s},\quad   \phi_s' = \phi_{1s},\quad (1\leqslant s\leqslant d),\label{3.7b}\end{equation}
we have $\Phi\not \equiv 0$. This contradicts   the $D$-observability of system \eqref{3.1}.

Conversely,  when rank$(D)=N$,  the  $D$-observation \eqref{3.5} implies that
\begin{equation}\partial_\nu\Phi\equiv \Phi\equiv 0\quad\hbox{ on }(0, T)\times \Gamma_1.\end{equation}  Then, Holmgren's uniqueness  theorem implies well  $\Phi\equiv 0$, provided that $T>0$ is large enough.
 \end{proof}

\begin{definition}\label{th4.2}  System  (\ref{4.1})  is approximately  null  controllable at the time $T>0$, if for any given initial data $(\widehat U_0,  \widehat U_1)\in ({\mathcal H}_0 )^N \times ({\mathcal H}_{-1})^N $, there exists a sequence $\{H_n\}$ of boundary controls in ${\mathcal L}^M$ with compact support in $[0,T]$, such that the sequence $\{U_n\}$ of solutions  to problem (\ref{4.1})-(\ref{4.2}) satisfies
\begin{align}
u_n^{(k)}\longrightarrow0
\quad   \text{ in  } C_{loc}^0 ([T,+\infty); {\mathcal H}_{0}) \cap C^1_{loc} ([T, +\infty); {\mathcal H}_{-1})
\end{align}
for all $1\leqslant k\leqslant N$ as $n\rightarrow  +\infty.$
\end{definition}By a similar argument as  in  \cite {AppSynDirichlet},  we can prove the following

\begin{proposition}\label{th4.3}
System (\ref{4.1})  is approximately   null controllable at the time $T>0$,  if and only if  its adjoint system
\eqref{3.1} is $D$-observable on the interval $[0,T]$.
\end{proposition}

\begin{corollary}\label{th4.4}  If  system (\ref{4.1}) is  approximately  controllable, then we necessarily have  $\hbox{rank}(\mathcal R)=N$.  In particular, as  $M=N$, namely, $D$ is  invertible,  system (\ref{4.1}) is  approximately null controllable.
\end{corollary}

\begin{proof}  This Corollary  follows immediately  from  Proposition \ref{th3.3} and Proposition \ref{th4.3}.   However,  here we prefer to give a direct proof  from the point of view of  control.

Suppose that $\hbox{dim Ker}(\mathcal R^T) =d\geqslant 1.$ Let Ker$(\mathcal R^T)= \hbox{Span}\{E_1,\cdots, E_d\}$.   By Lemma \ref{th2.1}, Ker$(\mathcal R^T)$ is contained in
Ker$(D^T)$ and invariant for both $A^T$ and $B^T$, then we still have  \eqref{3.6a} and  \eqref{3.6}.   Applying $E_r$ to problem  (\ref{4.1})-(\ref{4.2}) and setting
$u_r=(E_r, U)$ for $1\leqslant r\leqslant d,$ it follows that for $1\leqslant r\leqslant d,$ we have
\begin{equation}\left\lbrace
\begin{array}{ll}u_r'' -\Delta u_r + \sum_{s=1}^d\alpha_{rs}u_s
 =0&\hbox{ in }  (0, +\infty)\times\Omega,\\
u_r =0 & \hbox{ on }(0, +\infty)\times \Gamma_0,\\
\partial_\nu u_r +\sum_{s=1}^d\beta_{rs}u_s =0 & \hbox{ on }(0, +\infty)\times\Gamma_1\end{array}
\right.\label{4.5}
\end{equation}with the initial condition
\begin{equation}t=0:\quad  u_r=(E_r,  \widehat U_0), \quad
u_r'=(E_r,  \widehat U_1)\quad \hbox{in } \Omega.\label{4.6}\end{equation}
Thus,  the projections $u_1,\cdots, u_d$ of $U$ on the subspace Ker$(\mathcal R^T)$ are  independent of the applied boundary controls $H$, therefore, uncontrollable.
This contradicts the approximate  boundary null controllability of  system (\ref{4.1}). The proof is then complete.
\end{proof}

\vskip 1cm

\section{Uniqueness  of continuation}   By  Proposition \ref{th3.3},  $\hbox{rank}(\mathcal R) = N$ is a necessary condition for the $D$-observability.

\begin{proposition}   Let
\begin{equation}\mu = \sup_{\alpha, \beta\in \mathbb C} \hbox{dim Ker}\begin{pmatrix}A^T-\alpha I\\
B^T-\beta I\end{pmatrix}.\end{equation}Assume that
\begin{equation}\hbox{Ker}(\mathcal R^T) =\{0\}. \label{5.1}\end{equation}
Then we have the following lower bound estimate:
\begin{equation}\hbox{rank}(D)\geqslant  \mu.\label{5.2} \end{equation}
\end{proposition}

\begin{proof}   Let $\alpha,  \beta\in \mathbb C$,  such that
\begin{equation}V = \hbox{Ker}\begin{pmatrix}A^T-\alpha I\\
B^T-\beta I\end{pmatrix}\end{equation} is of dimension $\mu$.  It is easy to see that  any given subspace $W$ of $V$ is still invariant for  $A^T$ and $B^T$, then  by  Lemma \ref{th2.1}, condition \eqref{5.1}
implies that  $\hbox{Ker}(D^T)\cap V=\{0\}$. Then, it follows that
\begin{equation}\hbox{dim Ker}(D^T) + \hbox{dim } (V) \leqslant N, \end{equation}
namely,
\begin{equation}\mu=  \hbox{dim } (V) \leqslant N -\hbox{dim Ker}(D^T) =\hbox{rank}(D).\end{equation}
The proof is complete.\end{proof}

In general,  the condition $\hbox{dim Ker}(\mathcal R^T)=0$  does not imply $\hbox{rank}(D) =N$,
so, the $D$-observation \eqref{3.5} does not imply
\begin{equation}\Phi=0\quad\hbox{ on }(0, T)\times \Gamma_1.\end{equation}
Therefore,  the uniqueness of continuation for the solution to  the adjoint  system \eqref{3.1} with $D$-observatiuon \eqref{3.5} is not a standard type of Holmgren's uniqueness  theorem.
Up to now, we only  know fewer results on it,  which  we  outline  as  follows.

Consider the following Robin type mixed problem of  a  system of two equations
\begin{equation}\begin{cases}u''-\Delta u + au +bv=0 &\hbox{ in } (0, +\infty)\times \Omega,\\
v''-\Delta v + cu+dv =0 & \hbox{ in }  (0, +\infty)\times \Omega,\\
u=v=0 &\hbox{ on } (0, +\infty)\times \Gamma_0,\\\partial_\nu u +\alpha u=0 &\hbox{ on } (0, +\infty)\times  \Gamma_1,
\\\partial_\nu v + \beta v =0 &\hbox{ on } (0, +\infty)\times  \Gamma_1.\end{cases}\end{equation}
Here, since the boundary coupling matrix $B$  is assumed to be similar to a symmetric matrix, without loss of generality, we suppose that $B =diag(\alpha, \beta)$  is a diagonal matrix.
The following result can be easily  checked.
\begin{proposition} We have $\hbox{Ker}(\mathcal R^T)=\{0\}$  in the following cases.

(i) Case $\alpha\not = \beta$. Let $D=(d_1,  d_2)^T$.

\noindent(a) $d_1\not =0$,  if $(1, 0)^T$ is the only common eigenvector  of $A^T$ and $B^T$,

\noindent (b)  $d_2\not =0$,  if $(0, 1)^T$ is the only  common eigenvector  of $A^T$ and $B^T$,

\noindent (c)  $d_1d_2\not =0$, if both $(1, 0)^T$ and $(0, 1)^T$ are   eigenvectors of $A^T$ and $B^T$,

\noindent (d)  $d_1^2 + d_2^2\not =0$, if   there is no common  eigenvector for $A^T$ and $B^T$.

(ii)  Case $\alpha = \beta$.

\noindent (a)  $D= \mu_1 x_1 + \mu_2 x_2$  with $\mu_1\mu_2\not =0$, if $A$ possesses  two different eigenvalues, associated to two eigenvectors $x_1,  x_2$.

\noindent (b)  $D= \mu_1 x_1 + \mu_2 x_2$
with $\mu_1\not =0$, if $A$ possesses only one eigenvalue  associated to  an eigenvector $x_1$   and  a root vector $x_2$.
\end{proposition}

\begin{theoreme} (\cite{alabau-li-rao} Theorem 2.6)  Let $(u,  v)$ be a  solution to the following system of two equations:
\begin{equation}\begin{cases}u''-\Delta u=0 &\hbox{ in } (0, +\infty)\times \Omega,\\
v''-\Delta v +u=0 &\hbox{ in } (0, +\infty)\times \Omega,\\
u=v=0 &\hbox{ on } (0, +\infty)\times \Gamma_0,\\\partial_\nu u=
\partial_\nu v=0 &\hbox{ on } (0, +\infty)\times  \Gamma_1\end{cases}\label{5.4}\end{equation}
with initial data
in $ H^1_{\Gamma_0}(\Omega)\times H^1_{\Gamma_0}(\Omega)\times L^2(\Omega)\times L^2(\Omega)$. Then,  the observation
\begin{equation}\label{5.4b}d_1u + d_2v\equiv 0\quad \hbox{on }  [0, T]\times \Gamma_1\end{equation}
implies that  $u\equiv v\equiv 0$, provided that $d_2\not =0$ and $T>0$ is large enough. \end{theoreme}

\begin{theoreme}  (\cite{lux}) Let $(u,  v)$ be  a solution to the following system of  two equations:
\begin{equation}\begin{cases}u''-\Delta u =0 &\hbox{ in } (0, +\infty)\times \Omega,\\
v''-\Delta v=0 & \hbox{ in }  (0, +\infty)\times \Omega,\\
u=v=0 &\hbox{ on } (0, +\infty)\times \Gamma_0,\\\partial_\nu u +\alpha u=0 &\hbox{ on } (0, +\infty)\times  \Gamma_1,
\\ \partial_\nu v +\beta v=0 &\hbox{ on } (0, +\infty)\times  \Gamma_1\end{cases}\end{equation}
with initial data in
$ H^1_{\Gamma_0}(\Omega)\times H^1_{\Gamma_0}(\Omega)\times L^2(\Omega)\times L^2(\Omega)$.  Assume that
$\alpha\not =\beta$ and $d_1d_2\not =0$.  Then

(i)  In higher dimensional case, the observation in the infinite horizon:
\begin{equation} d_1u + d_2 v\equiv 0\quad\hbox{on }  (0, +\infty)\times \Gamma_1\end{equation}
implies that $u\equiv v\equiv 0$.

(ii) In one-space-dimensional case, the  observation in a finite horizon:
\begin{equation} d_1 u(1) + d_2 v(1)\equiv 0\quad\hbox{for }  0 \leqslant t\leqslant T\end{equation}
implies that $u\equiv v\equiv 0,$ provided that $T>0$ is large enough.
\end{theoreme}

Let us consider the following slightly modified  system:
\begin{equation}\begin{cases}u''-\Delta u=0 &\hbox{ in } (0, +\infty)\times \Omega,\cr
v''-\Delta v +u=0 &\hbox{ in } (0, +\infty)\times \Omega,\\
u=v=0 &\hbox{ on } (0, +\infty)\times \Gamma_0,\\\partial_\nu u +\alpha u=0 &\hbox{ on } (0, +\infty)\times  \Gamma_1,\\
\partial_\nu v+\beta v=0 &\hbox{ on } (0, +\infty)\times  \Gamma_1\end{cases}\label{5.5}\end{equation}
with  the partial  observation \eqref{5.4b} corresponding to $D=(d_1,  d_2)^T$.
By Lemma  \ref{th2.2} (ii),  $\hbox{Ker}(\mathcal R^T)=\{0\}$ if and only if  $\hbox{Ker}(D^T)$ does not contain any common eigenvector of $A^T$ and $B^T$. Since $(0, 1)^T$ is the only common eigenvector of $A^T$ and $B^T$,
$\hbox{Ker}(\mathcal R^T)=\{0\}$ if and only if  $(0, 1)^T\not \in \hbox{Ker}(D^T)$, namely,  if and only if $d_2\not =0.$ Unfortunately,
the  multiplier approach  used  in \cite{alabau-li-rao}  is  quite technically delicate, we  don't know up to now if it can be adapted
to get   the uniqueness of continuation for system  \eqref{5.5}
with  the partial observation \eqref{5.4b}.

\vskip 1cm

\section{ Approximate  boundary synchronization}

\begin{definition}  System  (\ref{4.1}) is approximately
synchronizable  at the time $T>0$,  if for any given  initial data $(\widehat U_0, \widehat U_1)\in
(\mathcal H_0)^N\times (\mathcal H_{-1})^N$, there exists a sequence
$\{H_n\}$ of boundary controls in $\mathcal L^M$
with compact support in $[0, T]$, such that the corresponding
sequence $\{U_n\}$ of solutions  to problem (\ref{4.1})-(\ref{4.2}) satisfies
\begin{equation}\label{6.1} u^{(k)}_n-u^{(l)}_n\rightarrow 0 \quad \hbox { in  } C^0_{loc}([T,+\infty);  \mathcal H_0)\cap  C^1_{loc}([T,+\infty);  \mathcal H_{-1})\end{equation}
for all $k, l$ with $1\leqslant k, l\leqslant N$ as $ n\rightarrow +\infty$.
\end {definition}

Define the synchronization matrix of order $(N-1)\times N$ by
\begin{equation}
 C_1=\begin{pmatrix}
1&-1\\
&1&-1\\
&&\ddots&\ddots\\
&&&1&-1\end{pmatrix}.
\end{equation}
Clearly,
\begin{equation}\text{Ker}(C_1) =\hbox{Span}\{e_1\} \hbox{  with }e_1= (1, \cdots, 1)^T.\label{6.2} \end{equation}
Then, the approximate boundary synchronization (\ref{6.1}) can be equivalently rewritten as
\begin{equation}C_1U_n\rightarrow 0\quad \hbox{ in  } (C^0_{loc}([T,+\infty);  \mathcal H_{0}))^{N-1} \cap  (C^1_{loc}([T,+\infty);  (\mathcal H_{-1}))^{N-1}\label{6.3} \end{equation}
as $n\rightarrow +\infty.$
\begin{definition}  The matrix $A$ satisfies the condition of $C_1$-compatibility, if  there exists a unique matrix $\overline A_1$ of order $(N-1)$, such that
 \begin{equation}C_1A=\overline A_1 C_1. \label{6.4} \end{equation}
The matrix $\overline A_1$ is called the reduced matrix of $A$ by $C_1$. \end{definition}

\begin{rem}
It was  shown in \cite{Wei} that the condition of $C_1$-compatibility  (\ref{6.4}) is equivalent to
\begin{equation}A\hbox{Ker}(C_1)\subseteq \hbox{Ker}(C_1).\label{6.5} \end{equation}
Then, noting (\ref{6.2}), the vector $e_1 =(1,\cdots, 1)^T$ is  an eigenvector of $A$,
corresponding to the eigenvalue $a$ given by
\begin{equation} a = \sum_{j=1}^Na_{ij},\quad i=1,\cdots, N.\label{6.6} \end{equation}
In \eqref{6.6}, $\sum_{j=1}^Na_{ij}$ is independent of   $i=1,\cdots, N$, called the raw-sum condition, which is also equivalent to the condition of $C_1$-compatibility \eqref{6.4} or \eqref{6.5}.

Similarly, the matrix $B$ satisfies the condition of $C_1$-compatibility, if  there exists a unique matrix $\overline B_1$ of order $(N-1)$, such that
 \begin{equation}C_1B=\overline B_1 C_1, \label{6.7} \end{equation}
which  is equivalent to the fact that
\begin{equation}B\hbox{Ker}(C_1)\subseteq \hbox{Ker}(C_1).\label{6.8} \end{equation}
Moreover,  the vector $e_1 =(1,\cdots, 1)^T$ is  also an eigenvector of $B$,
corresponding to the eigenvalue $b$ given by
\begin{equation} b = \sum_{j=1}^Nb_{ij},\quad i=1,\cdots, N,\label{6.9} \end{equation}
where the sum $\sum_{j=1}^Nb_{ij}$ is independent of   $i=1,\cdots, N$.\end{rem}

\begin{theoreme}   \label{th6.1}    Assume that system (\ref{4.1})  is  approximately  synchronizable.
Then we necessarily have $\hbox{rank}(\mathcal R)\geqslant N-1$.
\end{theoreme}

\begin{proof}  Otherwise, we have $\hbox{dim Ker}(\mathcal R^T)>1$. Let $\hbox{Ker}(\mathcal R^T)=\hbox{Span}\{E_1, \cdots, E_d\}$ with $d>1$. Noting that
\begin{equation}\hbox{dim Im}(C_1^T) +  \hbox{dim Ker}(\mathcal R^T)= N-1+d>N,
\end{equation}
there   exists an unit  vector $E\in  Im(C_1^T)\cap \hbox{Ker}(\mathcal R^T)$.
Let $E=C_1^Tx$ with $x\in \mathbb R^{N-1}$. The approximate boundary synchronization (\ref{6.3}) implies that
\begin{equation}(E, U_n) =(x, C_1U_n) \rightarrow 0  \quad  \hbox{ in  }  C^0_{loc}([T,+\infty);  \mathcal H_{0})\cap C^1_{loc}([T,+\infty);  \mathcal H_{-1})
\label{6.12} \end{equation}
as $n\rightarrow +\infty. $

On the other hand, since $E\in  \hbox{Ker}(\mathcal R^T)$, we have
\begin{equation} E=\sum_{r=1}^d\alpha_rE_r,\label{6.12c}\end{equation}
where the coefficients $\alpha_1, \cdots, \alpha_d$  are not all zero.
By Lemma \ref{th2.1}, Ker$(\mathcal R^T)$ is contained in
Ker$(D^T)$ and invariant for both $A^T$ and $B^T$, therefore we still have  \eqref{3.6a} and  \eqref{3.6}.
Thus,  applying $E_r$ to  problem (\ref{4.1})-(\ref{4.2}) and setting
$u_r = (E_r, U_n)$ for $1\leqslant r\leqslant d,$ we find again     problem (\ref{4.5})-(\ref{4.6}) with homogeneous boundary conditions. Noting  that problem (\ref{4.5})-(\ref{4.6}) is independent of   $n$, it follows from (\ref{6.12}) and \eqref {6.12c} that
\begin{equation} \sum_{r=1}^d\alpha_ru_r(T)\equiv \sum_{r=1}^d\alpha_ru_r'(T)\equiv0.\end{equation}
Then,  by well-posedness,  it is easy to see that

\begin{equation} \sum_{r=1}^d\alpha_r(E_r, \widehat U_0) \equiv \sum_{r=1}^d\alpha_r(E_r, \widehat U_1) \equiv 0\end{equation}
for an given    initial data $(\widehat U_0,  \widehat U_1)\in (\mathcal H_0)^N\times(\mathcal H_{-1})^N$.  This yields
\begin{equation} \sum_{r=1}^d\alpha_rE_r = 0. \end{equation}
Because of the linear independence of the vectors $E_1, \cdots, E_d$, we get a contradiction  $\alpha_1=\cdots =\alpha_d=0.$
\end{proof}

\begin{theoreme}   \label{th6.2}      Assume  that  system (\ref{4.1}) is approximately
synchronizable  under the minimum  $\hbox{rank}(\mathcal R)=N-1.$  Then,  we have the following assertions:

(i) There exists a vector $E_1 \in \hbox{Ker}(\mathcal R^T)$, such that  $(E_1, e_1)=1$  with $e_1=(1, 1, \cdots, 1)^T$.

(ii) For any given initial data $(\widehat U_0, \widehat U_1)\in
(\mathcal H_0)^N\times (\mathcal H_{-1})^N$, there exists a unique scalar function $u$ such that
\begin{equation} u^{(k)}_n\rightarrow u
\quad \hbox{ in  } C^0_{loc}([T,+\infty);   \mathcal H_{0})\cap  C^1_{loc}([T,+\infty); \mathcal H_{-1}) \label{6.3bb}\end{equation}
for all $1\leqslant k\leqslant N$ as $n\rightarrow +\infty.$

(iii) The matrices  $A$ and $B$  satisfy  the conditions of $C_1$-compatibility (\ref{6.4}) and (\ref{6.7}), respectively.
\end{theoreme}

\begin{proof} (i)  Noting that   $\hbox{dim Ker}(\mathcal R^T) =1,$  by Lemma \ref{th2.1}, there exists  a non-zero vector $E_1 \in \hbox{Ker}(\mathcal R^T)$,  such that
\begin{equation} D^TE_1 =0,\quad A^TE_1  = \alpha E_1 ,\quad B^T E_1 = \beta E_1.\end{equation}
We   claim  that $E_1 \not \in \hbox{Im}(C_1^T)$.
Otherwise, applying $E_1 $ to problem (\ref{4.1})-(\ref{4.2}) with $U=U_n$ and $H=H_n$,  and setting $u= (E_1 , U_n)$, it follows that
\begin{equation}\label{6.13} \left\lbrace
\begin{array}{ll}u'' -\Delta u +  \alpha u
 =0&\hbox{ in } (0, +\infty)\times\Omega,\cr
u =0 &\hbox{ on }  (0, +\infty)\times\Gamma_0,\cr
 \partial_\nu u +\beta u =0&\hbox{ on }(0, +\infty)\times \Gamma_1
 \end{array}
\right.
\end{equation}
with the following initial data
\begin{equation}\label{6.14}
t=0:\quad u =(E_1 , \widehat U_0), \quad  u'=(E_1 ,\widehat  U_1)\quad \hbox{in } \Omega.\end{equation}
Suppose that $E_1  \in \hbox{Im}(C_1^T)$, there exists a vector $x\in\mathbb R^{N-1}$,  such that $E_1 = C_1^Tx$. Then, the approximate boundary synchronization (\ref{6.3}) implies
 \begin{equation}
(u(T),  u' (T)) =((x, C_1U_n(T)), (x, C_1U_n'(T)))\rightarrow (0, 0)\end{equation}
in the space $\mathcal H_0\times \mathcal H_{-1}$ as $ n\rightarrow +\infty$. Since problem \eqref{6.13}-\eqref{6.14} is  independent of $n$, so  is  the solution $u$.  We  get thus 
\begin{equation}
u (T)\equiv u' (T)\equiv 0.\end{equation}
Thus, because of the well-posedness of problem (\ref{6.13})-(\ref{6.14}), it follows that
\begin{equation}
(E_1 , \widehat U_0)= (E_1 , \widehat U_1)=0\end{equation}
for any given  initial data $(\widehat U_0,  \widehat U_1)\in (\mathcal H_0)^N\times(\mathcal H_{-1})^N$.  This yields  a contradiction $E_1 =0$.

Since $E_1 \not \in \hbox{Im}(C_1^T)$, noting that $\hbox{Im}(C_1^T) = \hbox{Span}\{e_1 \}^\perp$, we have  $(E_1 , e_1 )\not =0$.
Without loss of generality, we can take $E_1 $ such that $(E_1 , e_1 )=1$.

(ii)  Since $E_1 \not \in Im(C_1^T)$, the matrix $
\begin{pmatrix} C_1\\ E_1 ^T \end{pmatrix}$ is invertible. Moreover, we have
\begin{equation}
\begin{pmatrix} C_1\\ E_1^T \end{pmatrix}e_1  =\begin{pmatrix} 0 \\ 1\end{pmatrix}.\label{6.a1}\end{equation}
Noting \eqref{6.3}, we have
\begin{equation}  \begin{pmatrix} C_1\\ E_1^T \end{pmatrix}U_n  = \begin{pmatrix} C_1U_n \\ (E_1,U_n)\end{pmatrix}\rightarrow
 \begin{pmatrix} 0\\
u\end{pmatrix} = u \begin{pmatrix} 0\\
1\end{pmatrix} \end{equation}
as $n\rightarrow +\infty$ in the space
\begin{equation}(C^0_{loc}([T,+\infty);  \mathcal H_0))^N\cap  (C^1_{loc}([T,+\infty);  \mathcal H_{-1}))^N. \label{6.a2} \end{equation}
Then, noting \eqref{6.a1}, it follows that
\begin{equation} U_n  = \begin{pmatrix} C_1\\ E_1^T \end{pmatrix}^{-1} \begin{pmatrix} C_1U_n \\ E_1^TU_n\end{pmatrix}\rightarrow u\begin{pmatrix} C_1\\ E_1^T \end{pmatrix}^{-1} \begin{pmatrix} 0\\
1\end{pmatrix}  =ue_1\label{6.15} \end{equation}
in the the space \eqref{6.a2}, namely,  \eqref{6.3bb} holds.

(iii) Applying $C_1$ to  system (\ref{4.1}) with $U=U_n$ and $H=H_n$,
and passing to the limit as $n\rightarrow +\infty$,  it follows from (\ref{6.3}) and   (\ref{6.15}) that
\begin{equation}C_1Ae_1 u=0 \quad \hbox{in } [T,+\infty)\times \Omega\label{6.16} \end{equation}
and
\begin{equation}C_1Be_1u=0 \quad \hbox{on }  [T,+\infty)\times \Gamma_1.\label{6.17} \end{equation}
We claim that at least for an initial data $(\widehat U_0,\widehat U_1)$, we have
\begin{equation}u\not\equiv0 \quad \hbox{on }  [T,+\infty)\times \Gamma_1.\end{equation}
Otherwise, it follows from system \eqref{6.13} that
\begin{equation}\partial_\nu u \equiv u\equiv0 \quad \hbox{on }  [T,+\infty)\times \Gamma_1,\end{equation}
then, by Holmgreen's uniqueness  theorem, we get
$u\equiv 0$ for all  the initial data $(\widehat U_0,\widehat U_1)$, namely, system (\ref{4.1}) is approximately null controllable under the condition $\hbox{dim Ker}(\mathcal R^T)=1.$ This contradicts  Corollary \ref{th4.4}. Then, it follows from (\ref{6.16})  and (\ref{6.17}) that
$C_1Ae_1=0$ and  $C_1Be_1=0$, which give the conditions of $C_1$-compatibility for $A$ and $B$, respectively. The proof is  complete.
\end{proof}

Assume that $A$ and $B$ satisfy the corresponding conditions  of $C_1$-compatibility, namely, there exist two matrices $\overline A_1$ and  $\overline B_1$ such that  $C_1A= \overline A_1C_1$ and  $C_1B= \overline B_1C_1$, respectively. Setting $W=C_1U$ in problem  (\ref{4.1})-(\ref{4.2}), we get the following reduced system
\begin{equation}\label {6.10} \left\lbrace
\begin{array}{ll} W''-\Delta W+ \overline A_1W=0 & \hbox{in }  (0, +\infty)\times \Omega, \\
W= 0& \hbox{on } (0, +\infty)\times \Gamma_0,\\
\partial_\nu W+ \overline B_1W=  C_1DH & \hbox{on }   (0, +\infty)\times\Gamma_1\end{array}
\right.
\end{equation}
with the initial condition
\begin{equation}\label{6.11} t=0:\quad W=C_1 \widehat U_0, \
W'=C_1\widehat U_1\quad \hbox{in }  \Omega.\end{equation}
Since $B$ is similar to a symmetric matrix, so is its reduced matrix $\overline B_1$ ({\it cf.} Proposition \ref{th7.2} below). Then,  by Proposition \ref{th5.1}, the  reduced problem (\ref{6.10})-(\ref{6.11}) is well-posed in the space $(\mathcal H_0)^{N-1}\times (\mathcal H_{-1})^{N-1}$.

Accordingly, consider the reduced adjoint system
\begin{equation}\begin{cases}\Psi''-\Delta \Psi+ \overline A_1^T\Psi=0&\hbox{ in }  (0, T)\times \Omega,\\
\Psi=0 & \hbox{ on }  (0, T)\times \Gamma_0,\\
\partial_\nu \Psi+\overline B_1^T\Psi=0 &
\hbox{ on }   (0, T)\times\Gamma_1\end{cases}\label{4.1a}\end{equation} with the $C_1D$-observation
\begin{equation}(C_1D)^T\Psi\equiv 0\quad \hbox{on } (0, T)\times \Gamma_1.\label{4.1b}\end{equation}
Obviously, we have
\begin{proposition}   \label{th6.a1}    Under the conditions of $C_1$-compatibility for $A$ and $B$, system (\ref{4.1})  is  approximately  synchronizable if and only if the reduced system (\ref{6.10}) is approximately  null controllable, or equivalently,  if and only if the reduced adjoint system (\ref{4.1a})  is $C_1D$-observable.
\end{proposition}

\begin{theoreme}   \label{th6.3}   Assume that   $A$ and $ B$ satisfy the conditions of $C_1$-compatibility (\ref{6.4}) and (\ref{6.7}), respectively.  Assume furthermore that $A^T$ and $ B^T$ admit a common eigenvector
$E_1$,  such that  $(E_1, e_1)=1$ with $e_1 =(1, \cdots, 1)^T$.  Let $D$  be defined by
\begin{equation}\hbox{Im}(D)=\hbox{Span}\{E_1\}^\perp.\label{6.c1}\end{equation}
Then  system (\ref{4.1}) is  approximate
synchronizable. Moreover, we have  $\hbox{rank}(\mathcal R)=N-1$.
\end{theoreme}
\begin{proof}
Since $(E_1, e_1)=1$, noting \eqref{6.c1}, we have $e_1\not \in  \hbox{Im}(D)$ and $\hbox{Ker}(C_1)\cap
\hbox{Im}(D)=\{0\}$. Therefore,  by  Lemma 2.2 in \cite {Cocv2017}, we have
\begin{equation}\hbox{rank}(C_1D) = \hbox{rank}(D)=N-1.\end{equation}
Thus,  the adjoint system \eqref{4.1a}
is $C_1D$-observable because of Holmgren's uniqueness theorem. By Proposition \ref{th6.a1}, system (\ref{4.1}) is  approximate
synchronizable.

Noting \eqref{6.c1}, we have $E_1\in \hbox{Ker}(D^T)$. Moreover, since $E_1$ is a common eigenvector of $A^T$ and $ B^T$, we  have   $E_1\in \hbox{Ker}(\mathcal R^T)$, hence   $\hbox{dim Ker}(\mathcal R^T)\geqslant 1$, namely, $\hbox{rank}(\mathcal R) \leqslant N-1$. On the other hand, since  $\hbox{rank}(\mathcal R)\geqslant \hbox{rank}(D)= N-1$, we get $\hbox{rank}(\mathcal R) =N-1$.
The proof is  complete. \end{proof}

\vskip 1cm

\section {Approximate boundary synchronization by $p$-groups}
In this section, let $p\geqslant 1$ be an integer and
\begin{equation}0=n_0< n_1<n_2<\cdots<n_p=N.\end{equation}We
rearrange the components of the state variable $U$ into $p$ groups:
\begin{equation}(u^{(1)},  \cdots,  u^{(n_1)}),
(u^{(n_1+1)}, \cdots, u^{(n_2)}), \cdots,  (u^{(n_{p-1}+1)},
\cdots, u^{(n_p)}).\end{equation}

\begin{definition}   \label{th7.1} System (\ref{4.1}) is approximately  synchronizable by $p$-groups at the time $T>0$,  if for any given initial data  $( \widehat U_0, \widehat U_1)\in (\mathcal H_0)^{N}\times (\mathcal H_{-1})^{N}$, there exists a sequence
$\{H_n\}$ of boundary controls in $\mathcal L^M$
with compact support in $[0, T]$, such that the corresponding
sequence $\{U_n\}$ of solutions  to problem (\ref{4.1})-(\ref{4.2}) satisfies
\begin{equation} u^{(k)}_n-u^{(l)}_n\rightarrow 0 \quad \hbox{ in } C^0_{loc}([T,+\infty);  \mathcal H_{0})\cap  C^1_{loc}([T,+\infty);  \mathcal H_{-1})\label{7.0}\end{equation}
  for   $n_{r-1}+1\leqslant k,l\leqslant  n_r$ and $1\leqslant r\leqslant p$
as $n\rightarrow +\infty.$
\end{definition}

Let $S_r$  be  the following $(n_r-n_{r-1}-1)\times
(n_r-n_{r-1})$  matrix
\begin{equation} S_r = \begin{pmatrix}1&-1&0&\cdots&0\\
0&1&-1&\cdots&0\cr\vdots&\vdots&\ddots&\ddots
&\vdots\\
0&0&\cdots&1&-1\\\end{pmatrix}.\end{equation} Let $C_p$ be the following  $(N-p)\times N$ full row-rank matrix of
synchronization by $p$-groups:
\begin{equation}C_p=\begin{pmatrix}S_1\\
&S_2\\ &&\ddots
\\ &&& S_p\end{pmatrix}.\end{equation}
For $1\leqslant r\leqslant p$,  setting
\begin{equation}(e_r)_j = \left\lbrace
\begin{array}{l} 1, \quad n_{r-1}+1\leqslant j\leqslant n_{r},\\
0,\quad \hbox{otherwise}.\end{array}
\right.\label{7.1} \end{equation}
It is clear that
\begin{equation}\hbox{Ker}(C_p)= \hbox{Span}\{e_1, e_2,\cdots, e_p\}.\label{e_p}\end{equation}
Moreover,  the approximate boundary synchronization by $p$-groups (\ref{7.0}) can be equivalently rewritten as
\begin{equation}C_pU_n\rightarrow 0\quad
\hbox{ in }  (C^0_{loc}([T,+\infty);  \mathcal H_{0}))^{N-p} \cap  (C^1_{loc}([T,+\infty);  (\mathcal H_{-1}))^{N-p}\label{7.2} \end{equation}
as $n\rightarrow +\infty.$

\begin{definition} The matrix $A$ satisfies the condition of $C_p$-compatibility,  if there exists a  unique matrix $\overline A_p$ of order $(N-p)$, such that
 \begin{equation}C_pA=\overline A_p C_p. \label{7.3} \end{equation}
The matrix $\overline A_p$  is called the reduced matrix of $A$ by $C_p$.
\end{definition}

\begin{rem}
The  condition of $C_p$-compatibility (\ref{7.3}) is equivalent to
\begin{equation}A\hbox{Ker}(C_p)\subseteq \hbox{Ker}(C_p).\label{7.4} \end{equation}
Moreover, the reduced matrix $\overline A_p$ is given by
\begin{equation}\overline A_p = C_pAC_p^T(C_pC_p^T)^{-1} \label{7.4b} \end{equation}
(see  Lemma 3.3 in  \cite{Wei}). Similarly,
 the matrix $B$ satisfies the condition of $C_p$-compatibility, if  there exists a unique matrix $\overline B_p$ of order $(N-p)$, such that
 \begin{equation}C_pB=\overline B_p C_p, \label{7.5} \end{equation}
which  is equivalent to
\begin{equation}B\hbox{Ker}(C_p)\subseteq \hbox{Ker}(C_p).\label{7.6} \end{equation}
\end{rem}

\begin{proposition}\label{th7.2}
Assume that  $A$ satisfies the  condition of $C_p$-compatibility (\ref{7.3}).   Let  $\{x_l^{(k)}\}_{ 1\leqslant k\leqslant d, 1\leqslant l\leqslant r_k}$ be a system of root vectors  of the matrix $A$, corresponding to the eigenvalues $\lambda_k \ (1\leqslant k\leqslant d)$, such that for each $k \ (1\leqslant k\leqslant d)$ we have
 \begin{equation} \label{7.20b}Ax_l^{(k)}= \lambda_kx_l^{(k)} +  x_{l+1}^{(k)},\quad   1\leqslant l\leqslant r_k \hbox{ with } x_{r_k+1}^{(k)} =0.\end{equation}
Define the following projected vectors by
\begin{equation}\label{7.21b} \overline x_l^{(k)} = C_p x_l^{(k)},\quad 1\leqslant k\leqslant \overline d,\quad 1\leqslant l\leqslant\overline  r_k,
\end{equation}
where $\overline d\ (1\leqslant \overline d\leqslant d)$ and $\overline r_k \  (1\leqslant \overline r_k\leqslant r_k)$  are given by (\ref{7.22}) below. Then $\{\overline x_l^{(k)}\}_{1\leqslant k\leqslant\overline d, 1\leqslant l\leqslant \overline r_k}$ forms  a system of root vectors  of the reduced matrix $\overline A_p$. In particular, if $A$ is similar to a symmetric matrix, then so is $\overline A_p$.
\end{proposition}

 \begin{proof}Since Ker$(C_p)$ is an invariant subspace of $A$, without loss of generality,
we may assume that  there exist some integers $\overline d \   (1\leqslant \overline d\leqslant d)$ and $\overline r_k \ (1\leqslant \overline r_k\leqslant r_k)$, such that the  $\{x_l^{(k)}\}_{1\leqslant k\leqslant \overline d, 1\leqslant l\leqslant \overline r_k}$ forms a root system for the  restriction of $A$ on the invariant subspace Ker$(C_p)$. Then,
\begin{equation}\label{7.22}\hbox{Ker}(C_p) = \hbox{Span}\{x_l^{(k)}: \  1\leqslant k\leqslant \overline d,  \ 1\leqslant l\leqslant \overline r_k\}.
\end{equation}
In particular, we have
\begin{equation} \label{7.23}\sum_{k=1}^{\overline d} (r_k-\overline r_k) =p.
\end{equation}

Noting that $C_p^T(C_pC_p^T)^{-1}C_p$ is a projection from $\rr ^N$ onto $\hbox{Im}(C_p^T)$,  we have
\begin{equation}
C_p^T(C_pC_p^T)^{-1}C_p x=x,\quad \forall x\in \hbox{Im}(C_p^T).\label{7.24}
\end{equation}

On the other hand, by $\rr^N= \hbox{Im}(C_p^T)\oplus \hbox{Ker}(C_p)$ we can write
 \begin{equation} \label{7.25}x_l^{(k)}= \widehat x_l^{(k)} + \widetilde x_l^{(k)}\quad \hbox{with}\quad \widehat x_l^{(k)} \in \hbox{Im}(C_p^T),\quad
\widetilde  x_l^{(k)}\in \hbox{Ker}(C_p),\end{equation}
then it follows from (\ref{7.21b}) that
\begin{equation}\label{7.26} \overline x_l^{(k)} = C_p \widehat x_l^{(k)},\quad 1\leqslant k\leqslant \overline d,\quad 1\leqslant l\leqslant\overline  r_k.
\end{equation}
Thus, noting  (\ref{7.4b}) and (\ref{7.24}), we have
\begin{equation}
 \overline A_p  \overline x_l^{(k)}=C_pAC_p^T(C_pC_p^T)^{-1}C_p
\widehat x_l^{(k)}
=C_pA\widehat x_l^{(k)}.\end{equation}
Since Ker$(C_p)$ is invariant for $A$, $A\widetilde x_l^{(k)}\in \hbox{Ker}(C_p)$, then $C_pA\widetilde x_l^{(k)}=0$. It follows that

\begin{equation}
 \overline A_p  \overline x_l^{(k)}=C_pA(\widehat x_l^{(k)} +\widetilde x_l^{(k)})
=C_pAx_l^{(k)}.\end{equation}
Then, using (\ref{7.20b}) and (\ref{7.21b}), it is easy to see that

\begin{equation}
 \overline A_p  \overline x_l^{(k)}=C_p (\lambda_kx_l^{(k)} +x_{l+1}^{(k)})= \lambda_k\overline x_l^{(k)} +\overline x_{l+1}^{(k)}.
\end{equation}
Therefore, $\overline x_1^{(k)}, \overline x_2^{(k)}, \cdots,  \overline x_{\bar r_k}^{(k)}$ is  a Jordan  chain with length $\bar r_k$ of the reduced matrix $\overline A_p$,  corresponding to the eigenvalue $\lambda_k$.

 Since dim Ker$(C_p)=p$, the projected  system  $\{\overline x_l^{(k)}\}_{1\leqslant k\leqslant  \overline d, 1\leqslant  l\leqslant  \overline r_k}$ is of rank $(N-p)$. On the other hand, by (\ref{7.23}), system $\{\overline x_l^{(k)}\}_{1\leqslant k\leqslant \overline d, 1\leqslant l\leqslant \overline r_k}$ contains $(N-p)$ vectors, therefore, forms  a system of root vectors  of the reduced matrix $\overline A_p$. The proof is complete. \end{proof}

Assume that $A$ and $ B$ satisfy the  conditions of $C_p$-compatibility (\ref{7.3}) and (\ref{7.5}), respectively.  Setting $W=C_pU$ in problem  (\ref{4.1})-(\ref{4.2}), we get the following reduced system:
\begin{equation}\label {7.7} \left\lbrace
\begin{array}{ll} W''-\Delta W+ \overline A_pW=0& \hbox{ in }  (0, +\infty)\times \Omega, \\
W= 0& \hbox{ on } (0, +\infty)\times \Gamma_0,\\
\partial_\nu W+ \overline B_pW=  C_pDH & \hbox{ on }   (0, +\infty)\times\Gamma_1\end{array}
\right.
\end{equation}
with the initial condition
\begin{equation}\label{7.8} t=0:\quad W=C_p \widehat U_0, \quad
W'=C_p\widehat U_1\quad \hbox{in }  \Omega.\end{equation}

Since $B$ is similar to a symmetric matrix, by Proposition \ref{th7.2}, the  reduced matrix $\overline B_p$ is also similar to a symmetric matrix. Then by Proposition \ref{th5.1} and Remark \ref{th4.2aa}, the  reduced problem (\ref{7.7})-(\ref{7.8}) is well-posed in the space $(\mathcal H_0)^{N-p}\times (\mathcal H_{-1})^{N-p}$.

Accordingly, consider the reduced adjoint system
\begin{equation}\begin{cases}\Psi''-\Delta \Psi+ \overline A_p^T\Psi=0&\hbox{ in }  (0, +\infty)\times \Omega,\\
\Psi=0 & \hbox{ on }  (0, +\infty)\times \Gamma_0,\\
\partial_\nu \Psi+\overline B_p^T\Psi=0 &
\hbox{ on }   (0, +\infty)\times\Gamma_1\end{cases}\label{7.7b}\end{equation}
together with  the $C_pD$-observation
\begin{equation}(C_pD)^T\Psi\equiv 0\quad \hbox{ on } (0, T)\times \Gamma_1.\end{equation}

We have
\begin{proposition}   \label{th7.a1} Assume that $A$ and $ B$ satisfy the  conditions of $C_p$-compatibility (\ref{7.3}) and (\ref{7.5}), respectively.    Then system (\ref{4.1})  is  approximately  synchronizable by $p$-groups if and only if the reduced system (\ref{7.7}) is approximately  null controllable, or equivalently, if and only if the reduced adjoint system (\ref{7.7b})  is $C_pD$-observable.
\end{proposition}

\begin{corollary}   \label{th10.q} Under the conditions of $C_p$-compatibility  (\ref{7.3}) and (\ref{7.5}), if  system  (\ref{4.1})   is  approximately  synchronizable by $p$-groups, we necessarily have the following  rank condition:
\begin{equation}\hbox{rank}(C_p\mathcal R)= N-p.\label{10.8f}\end{equation}
\end{corollary}
\begin{proof} Let $\overline{ \mathcal R}$  be the matrix defined by  \eqref{2.1}-\eqref{2.2} corresponding to the reduced matrices $\overline A_p, \overline B_p$ and $\overline D=C_pD$.  Noting \eqref{7.3}  and  \eqref{7.5}, we have
\begin{equation}\overline A_p^r\overline B_p^s\overline D=\overline A_p^r\overline B_p^sC_pD = C_pA^rB^sD,\end{equation}
then
\begin{equation}\overline{ \mathcal R} = C_p\mathcal R.\label{bb}\end{equation}

Under the assumption  that system  (\ref{4.1})   is  approximately  synchronizable by $p$-groups, by Proposition \ref{th7.a1},   the reduced system (\ref{7.7}) is approximately null controllable, then by Corollary \ref{th4.4}, we have
rank $(\overline{ \mathcal R})= N-p$  which together with \eqref{bb}, implies \eqref{10.8f}.
  \end{proof}

\begin{proposition}  \label{th7.3} Assume that system (\ref{4.1}) is  approximately   synchronizable by $p$-groups. Then,  we necessarily have  $\hbox{rank}(\mathcal R)\geqslant N-p$.\end{proposition}

\begin{proof} Assume   $\hbox{dim Ker}(\mathcal R^T)=d$ with $d>p$.
Let   $\hbox{Ker}(\mathcal R^T) =\hbox{Span}\{E_1, \cdots, E_d\}$.
Since
$$\hbox{dim Ker}(\mathcal R^T) + \hbox{dim Im}(C_p^T) = d+N-p >N,$$
we have  $\hbox{Ker}(\mathcal R^T)\cap\hbox{Im}(C_p^T)\not =\{0\}$. Hence, there exists a non-zero vector $x\in \mathbb R^{N-d}$ and  coefficients $\beta_1,\cdots, \beta_d$ not all zero,  such that
\begin{equation}\sum_{r=1}^d\beta_rE_r= C_p^Tx.\label{7.9}
\end{equation}
Moreover, by Lemma \ref{th2.1}, we still  have  \eqref{3.6a} and  \eqref{3.6}. Then, applying $E_r$ to problem (\ref{4.1})-(\ref{4.2}) with $U=U_n$ and $H=H_n$ and setting
$u_r=(E_r, U_n) $ for $ 1\leqslant r\leqslant d,$
 it follows that
\begin{equation}\left\lbrace
\begin{array}{ll}u_r'' -\Delta u_r + \sum_{s=1}^d\alpha_{rs}u_s
 =0&\hbox{ in }  (0, +\infty)\times\Omega,\\
u_r =0 & \hbox{ on }(0, +\infty)\times \Gamma_0,\\
\partial_\nu u_r +\sum_{s=1}^d\beta_{rs}u_s =0 & \hbox{ on }(0, +\infty)\times\Gamma_1\end{array}
\right.\label{7.10}
\end{equation}with the initial condition
\begin{equation}t=0:\quad  u_r=(E_r,   \widehat U_0),  \
u_r'=(E_r,  \widehat U_1)\quad \hbox{in } \Omega.\label{7.11} \end{equation}
Noting (\ref{7.2}), it follows from  (\ref{7.9}) that
\begin{equation}\sum_{r=1}^d\beta_ru_r=(x, C_pU_n) \rightarrow 0  \quad \hbox{ in  }C^0_{loc}([T,+\infty);  \mathcal H_{0})\cap C^1_{loc}([T,+\infty);  \mathcal H_{-1})\end{equation}
as $n\rightarrow +\infty.$
Since problem \eqref{7.10}-\eqref{7.11} is independent of $n$, so is the solution $(u_1,\cdots, u_d)$.  It follows that
\begin{equation}\sum_{r=1}^d\beta_ru_r(T)=\sum_{r=1}^d\beta_ru_r'(T)=0\quad \hbox{ in }  \Omega.\end{equation}
Then,  it follows from the well-posedness of problem (\ref{7.10})-(\ref{7.11}) that
\begin{equation} \sum_{r=1}^d\beta_r(E_r, \widehat U_0)=\sum_{r=1}^d\beta_r(E_r, \widehat U_1) =  0\end{equation}
for any given  initial data $(\widehat U_0,\widehat  U_1)\in (\mathcal H_0)^N\times(\mathcal H_{-1})^N $. In particular, we get
\begin{equation} \sum_{r=1}^d\beta_rE_r= 0,\end{equation}
then a contradiction: $\beta_1=\cdots=\beta_d =0$,  because of the linear independence of the vectors $E_1,\cdots, E_d.$
The proof is achieved.
\end{proof}

\begin{theoreme} \label{th7.4} Let  $A$ and $ B$ satisfy the  conditions of $C_p$-compatibility (\ref{7.3}) and (\ref{7.5}), respectively.
Assume that $A^T$ and $B^T$ admit a common invariant subspace $V,$ which is bi-orthonormal to $Ker(C_p)$.
Then, setting  the boundary control matrix  $D$ by
 \begin{equation}\hbox{Im}(D)= V^\perp,\end{equation}
 system (\ref{4.1}) is  approximately
synchronizable by $p$-groups.  Moreover, we have $\hbox{rank}(\mathcal R)= N-p$.
\end{theoreme}
\begin{proof}

Since $V$ is bi-orthonormal to $\hbox{Ker}(C_p)$, we have
\begin{equation}\hbox{Ker}(C_p)\cap V^\perp = \hbox{Ker}(C_p)\cap \hbox{Im}(D)
=\{0\},\end{equation}  therefore, by Lemma 2.2 in \cite {Cocv2017}, we have
\begin{equation}\hbox{rank}(C_pD) = \hbox{rank}(D)=N-p.\end{equation}
Thus,  the $C_pD$-observation (\ref{4.1b}) becomes the full observation
\begin{equation}\Psi\equiv 0\quad \hbox{ on } (0, T)\times \Gamma_1.\end{equation}
By  Holmgren's uniqueness theorem, the reduced adjoint system \eqref{7.7b} is  observable and the reduced system \eqref{7.7}  is approximately null controllable. Then, by Proposition \ref {th7.a1}, the original system (\ref{4.1}) is  approximately
synchronizable by $p$-groups.  Noting that $\hbox{Ker}(D^T) = V$, by Lemma \ref{th2.1}, it is easy to see that  $\hbox{rank}(\mathcal R)= N-p$. The proof is then complete.
\end{proof}

\begin{theoreme} \label{th7.5}    Assume that  system   (\ref{4.1}) is approximately
synchronizable by $p$-groups.  Assume furthermore that  $\hbox{rank}(\mathcal R) =N-p$. Then, we have the following assertions:

(i) $\hbox{Ker}(\mathcal R^T)$ is bi-orthonormal  to $\hbox{Ker}(C_p)$.

(ii) For any given initial data  $(\widehat U_0, \widehat U_1)\in
(\mathcal H_0)^N\times (\mathcal H_{-1})^N$, there exist  unique   scalar  functions $u_1, u_2,\cdots, u_p$
such that
\begin{equation}u^{(k)}_n\rightarrow u_r \quad \hbox{ in }  C^0_{loc}([T,+\infty);  \mathcal H_0)\cap  C^1_{loc}([T,+\infty);
H^{-1}(\Omega))\label{7.a4}
\end{equation}  for $n_{r-1}+1\leqslant k\leqslant n_r \hbox{ and }
1\leqslant r\leqslant p$ as $n\rightarrow +\infty.$

(iii) The coupling matrices $A$ and $B$ satisfy the  conditions of $C_p$-compatibility (\ref{7.3}) and (\ref{7.5}), respectively.
\end{theoreme}

\begin{proof}  (i) We   claim that $\hbox{Ker}(\mathcal R^T)\cap \hbox{Im}(C_p^T) =\{0\}$. Then,  noting that  $\hbox{Ker}(\mathcal R^T)$ and Ker$(C_p)$ have the same dimension $p$ and
\begin{equation} \hbox{Ker}(\mathcal R^T)\cap \{\hbox{Ker}(C_p)\}^\perp = \hbox{Ker}(\mathcal R^T)\cap \hbox{Im}(C_p^T) =\{0\},\end{equation}
 by  Proposition 4.1 in \cite{JMPA2016},  $\hbox{Ker}(\mathcal R^T)$ and Ker$(C_p)$ are bi-orthonormal.
Then, let $\hbox{Ker}(\mathcal R^T) = \hbox{Span}\{E_1, \cdots,  E_p\}$ and $\hbox{Ker}(C_p) = \hbox{Span}\{e_1, \cdots,  e_p\}$ such that
\begin{equation} (E_r, e_s) =\delta_{rs}, \quad r, s=1,\cdots, p.\label{7.12} \end{equation}

Now we return to check  that  $\hbox{Ker}(\mathcal R^T)\cap \hbox{Im}(C_p^T) =\{0\}$.   If $\hbox{Ker}(\mathcal R^T)\cap \hbox{Im}(C_p^T) \not =\{0\}$, there exist a non-zero vector $x \in \mathbb R^{N-p}$  and  some coefficients $\beta_1,\cdots,\beta_p$ not all zero, such that
\begin{equation}\sum_{r=1}^p\beta_rE_r= C_p^Tx.\label{7.13}
\end{equation}
By Lemma \ref{th2.1}, we still have \eqref{3.6a} and  \eqref{3.6} with $d=p$.
For $1\leqslant r\leqslant p$, applying $E_r$  to problem (\ref{4.1})-(\ref{4.2}) with $U= U_n$ and $H= H_n$, and  setting
\begin{equation} u_r= (E_r, U),\label{u_r}\end{equation}  it follows that
\begin{equation}\label{7.14} \left\lbrace
\begin{array}{ll}u_r'' -\Delta u_r +  \sum_{s=1}^p\alpha_{rs}u_s
 =0&\hbox{ in } (0, +\infty)\times\Omega,\cr
 u_r =0 &\hbox{ on }  (0, +\infty)\times\Gamma_0,\cr
 \partial_\nu u_r +  \sum_{s=1}^p\beta_{rs}u_s =0&\hbox{ on }(0, +\infty)\times \Gamma_1\end{array}
\right.
\end{equation}
with the initial condition
\begin{equation}t=0:\quad u_r = (E_r, \widehat U_0), \  u_r '=   (E_r, \widehat U_1).\label{7.15}
\end{equation}

Noting (\ref{7.2}), we have
\begin{equation}\sum_{r=1}^p\beta_ru_r =(x, C_pU_n) \rightarrow 0  \hbox{ in }  C^0_{loc}([T,+\infty);  \mathcal H_{0})\cap C^1_{loc}([T,+\infty);  \mathcal H_{-1})
\end{equation}
AS $n\rightarrow +\infty.$

Since the functions $u_1,\cdots, u_p$ are independent of $n$ and of  the applied boundary controls,  we have
\begin{equation} \sum_{r=1}^p\beta_ru_r(T)\equiv \sum_{r=1}^p\beta_ru_r'(T)\equiv0\quad \hbox{ in } \Omega.\end{equation}
Then,  it follows from the well-posedness of problem (\ref{7.14})-(\ref{7.15}) that
\begin{equation} \sum_{r=1}^p\beta_r(E_r, \widehat U_0) = \sum_{r=1}^p\beta_r(E_r, \widehat U_1)=0\quad \hbox{in } \Omega\end{equation}
for any given initial data $(\widehat U_0, \widehat U_1)\in (\mathcal H_0)^N\times(\mathcal H_{-1})^N $. In particular, we get
\begin{equation} \sum_{r=1}^p\beta_rE_r= 0,\end{equation}
then,  a contradiction: $\beta_1=\cdots=\beta_p =0$,  because of the linear independence of the vectors $E_1,\cdots, E_p.$

(ii) Noting  (\ref{7.2}), we have\begin{equation} \begin{pmatrix} C_p\\ E_1^T\\\cdot\\E_p^T \end{pmatrix}U_n
= \begin{pmatrix} C_pU_n \\ E_1^TU_n \\\cdot\\E_p^TU_n \end{pmatrix}\rightarrow \begin{pmatrix} 0\\ u_1\\\cdot\\ u_p \end{pmatrix} \label{7.16} \end{equation}
as $ n\rightarrow +\infty$ in the space  \begin{equation}(C^0_{loc}([T,+\infty);  \mathcal H_0))^N\cap  (C^1_{loc}([T,+\infty);  \mathcal H_{-1}))^N,  \label{7.a3}\end{equation} where $u_1, \cdots, u_p$ are given by \eqref{7.14}. Since $\hbox{Ker}(\mathcal R^T)\cap \hbox{Im}(C_p^T) =\{0\},$ the matrix
$\begin{pmatrix} C_p\\ E_1^T\\\cdot\\E_p^T \end{pmatrix}$ is invertible. Thus it follows from (\ref{7.16}) that there exists   $U$ such that
\begin{equation} U_n\rightarrow \begin{pmatrix} C_p\\ E_1^T\\\cdot\\E_p^T \end{pmatrix}^{-1}\begin{pmatrix} 0\\ u_1\\\cdot\\ u_p \end{pmatrix} =:U\label{7.17} \end{equation}
as $n\rightarrow +\infty$ in the space \eqref{7.a3}.
Moreover,  \eqref{7.2} implies that
$$t\geqslant T:\quad C_pU\equiv 0\quad \hbox{in } \Omega.$$
Noting   \eqref{e_p}, \eqref{7.12}  and \eqref{u_r}, it  follows that
\begin{equation}\label{7.19} t\geqslant T:\quad U = \sum_{r=1}^p(E_r, U)e_r = \sum_{r=1}^pu_re_r\quad \hbox{in } \Omega.\end{equation}
Noting \eqref{7.1}, we get then  \eqref{7.a4}.

(iii) Applying $C_p$ to  system (\ref{4.1}) with $U=U_n$ and $H=H_n$,
and passing to the limit as $n\rightarrow +\infty$,  by (\ref{7.2}), (\ref{7.17}) and (\ref{7.19}), it is easy to get
that \begin{equation} \sum_{r=1}^pC_pAe_ru_r(T)\equiv 0\quad \hbox{in }  \Omega\label{1} \end{equation}
and \begin{equation} \sum_{r=1}^pC_pBe_ru_r(T)\equiv 0\quad \hbox{on }   \Gamma_1.\label{2}  \end{equation}
Since system \eqref{7.14} is well-posed   in $(\mathcal H_1)^p\times (\mathcal H_0)^p $ and   time-invertible, so it  defines  an isomorphism  from
$(\mathcal H_1)^p\times (\mathcal H_0)^p $ onto $(\mathcal H_1)^p\times (\mathcal H_0)^p $. On the other hand,  the mapping 
\begin{equation} (U_0, U_1) \rightarrow ((E_r, \widehat U_0),   (E_r, \widehat U_1))_{1\leqslant r \leqslant p}\end{equation}
is surjective from $(\mathcal H_1)^N\times (\mathcal H_0)^N $ onto 
$(\mathcal H_1)^p\times (\mathcal H_0)^p $. Then, $(u_1, \cdots, u_p)$ will fulfil the space $(\mathcal H_1)^p\times (\mathcal H_0)^p $ as the initial data $(U_0, U_1)$ runs through the space $(\mathcal H_1)^N\times (\mathcal H_0)^N $.
There exist thus an initial date  $(U_0, U_1)\in (\mathcal H_1)^N\times (\mathcal H_0)^N$ such that the corresponding $(u_1(T), \cdots, u_p(T))$ are linearly independent.  Then, it follows from \eqref{1} and \eqref{2}  that
\begin{equation} C_pAe_r=0 \quad\hbox{and} \quad C_pBe_r=0\quad \hbox{ for } 1\leqslant r\leqslant p. \end{equation} We get thus the conditions of $C_p$-compatibility for $A$ and $B$,  respectively. The proof is  complete.
\end{proof}

\begin{rem} The convergence \eqref{7.a4}  will  be   called  the approximate boundary   synchronization by $p$-groups in the pinning sense, and   $(u_1, \cdots, u_p)^T$  will be called the approximately  synchronizable state by $p$-groups. 
While the convergence \eqref{7.2}  given by  Definition \ref{th7.1} will be  called the approximate boundary synchronization by $p$-groups in the consensus  sense.

In general,  the convergence (\ref{7.2})  does not imply the convergence \eqref{7.a4}.  In fact, we even don't know if  the sequence $\{U_n\}$ is bounded. However, under the rank condition  $\hbox{rank}(\mathcal R) =N-p$,  the convergence (\ref{7.2})
actually implies the convergence \eqref{7.a4}.   Moreover,  the fonctions $u_1, \cdots, u_p$  are independent of applied boundary controls. 
\end{rem}

Let $\mathbb D_p$ be the  set of all the boundary control matrices $D$ which realize the approximate boundary synchronization by $p$-groups  for  system (\ref{4.1}).
In order to show  the  dependence on $D$,  we prefer to write
$\mathcal R_D$ instead of $\mathcal R$  in \eqref{2.2}. Then, we may define the minimal rank as
\begin{equation}N_p =  \inf_{D\in \mathbb D_p}\hbox{rank} (\mathcal R_D).\end{equation}
Noting that $\hbox{rank}(\mathcal R_D)= N-\hbox{dim Ker}(\mathcal R_D^T)$, because of Proposition \ref{th7.3}, we have
\begin{equation}N_p\geqslant N-p.\label{7.20} \end{equation}
Moreover, we have  the following

\begin{corollary} The equality
\begin{equation}N_p=N-p\label{7.21} \end{equation}
holds  if and only if the coupling matrices $A$ and $B$ satisfy the  conditions of $C_p$-compatibility (\ref{7.3}) and (\ref{7.5}), respectively and   $A^T, B^T$  possess a common invariant  subspace,  which  is bi-orthonormal  to  Ker$(C_p)$. Moreover,
the approximate  synchronization is  in the pinning sense.
\end{corollary}

\begin{proof}   Assume that  \eqref{7.21}  holds. Then there exists a matrix $D\in\mathbb D_p$, such that
$\hbox{dim Ker}(\mathcal R_D^T)=p$. By  Theorem  \ref{th7.5},  the coupling matrices $A$ and $B$ satisfy the  conditions of $C_p$-compatibility (\ref{7.3}) and (\ref{7.5}), respectively,  and  $\hbox{Ker}(\mathcal R_D^T)$ which,  by Lemma \ref{th2.1},   is  bi-orthonormal  to  Ker$(C_p)$, is invariant  for both  $A^T$ and $B^T$. Moreover, the approximate  synchronization is  in the pinning sense.

Conversely,   let  $V$ be a subspace, which is invariant for both $A^T$ and $B^T$,  and bi-orthonormal  to  Ker$(C_p)$. Noting that $A$ and $B$ satisfy the  conditions of $C_p$-compatibility (\ref{7.3}) and (\ref{7.5}),  respectively, by Theorem \ref {th7.4},
there exists a matrix $D\in \mathbb D_p$,  such that   $\hbox{dim Ker}(\mathcal R_D^T)=p$, which together with \eqref{7.20}  implies \eqref{7.21}. \end{proof}

\begin{rem} If $N_p>N-p$, then the situation is more complicated. We don't know if  the conditions of $C_p$-compatibility (\ref{7.3}) and (\ref{7.5}) are necessary, either if the approximate boundary  synchronization by $p$-groups is  in the pinning sense.
\end{rem}

\vskip 1cm

\section{Approximately  synchronizable state by $p$-groups}

In  Theorem \ref{th7.5},  we have shown that if system \eqref{4.1} is approximately synchronizable by $p$-groups under the condition $\hbox{dim Ker}(\mathcal R^T) = p$,
then $A$ and $ B$ satisfy the corresponding conditions of $C_p$-compatibility,  and $\hbox{Ker}(\mathcal R^T)$ is bi-orthonormal to Ker$(C_p)$, moreover,
the  approximately  synchronizable state by $p$-groups   is  independent of the applied boundary controls. The following is the  counterpart.

\begin{theoreme}   \label{th8.1} Let   $A$ and $ B$ satisfy  the  conditions of $C_p$-compatibility (\ref{7.3}) and (\ref{7.5}), respectively.  Assume that   system (\ref{4.1})  is   approximately  synchronizable by $p$-groups.  If the projection of  any solution $U$ to  problem \eqref{4.1}-\eqref{4.2}  on a subspace $V$ of dimension $p$ is independent of applied boundary controls, then
$V=\hbox{Ker}(\mathcal R^T)$. Moreover,   $\hbox{Ker}(\mathcal R^T)$ is  bi-orthonormal  to  $\hbox{Ker}(C_p)$.\end{theoreme}

\begin{proof}   Fixing $\widehat U_0=\widehat U_1=0$, by  Proposition \ref{th5.1}, the  linear map
$$F:\quad H\rightarrow  U$$
is continuous, therefore,  infinitely differential
from the control space $\mathcal L^M$ to the space $C_{loc}^0([0, +\infty); (\mathcal H_0)^N)\cap C_{loc}^1([0, +\infty); (\mathcal H_{-1})^N).$  

Let
$\widehat U$ be defined by 
$$\widehat U = F'(0)\widehat H,$$ 
where $F'(0)$ is the Fr\^echet differential  of $F$, and    $\widehat H\in  \mathcal  L^M$ is any given boundary control.

Then, by linearity we have
\begin{equation}\left\lbrace
\begin{array}{ll}\widehat U''-{\Delta} \widehat U+A\widehat U=0&\hbox{in }  (0, +\infty)\times \Omega,\\
\widehat  U= 0 &\hbox{on }  (0, +\infty)\times\Gamma_0,\\
\partial_\nu \widehat U + B\widehat U = D\widehat  H&\hbox{on }  (0, +\infty)\times\Gamma_1,\\
t=0:\quad \widehat U =\widehat U'=0&\hbox{in }   \Omega.\end{array}
\right.\label{8.1}
\end{equation}
Let $V = \hbox{Span}\{E_1, \cdots,  E_p\}$. Then,  the  independence   of the projection of $U$ on the subspace $V$, with respect to the boundary  controls,  implies that
\begin{equation}(E_i, \widehat U)\equiv 0 \quad \hbox{ in } (0, +\infty) \times \Omega \quad \hbox { for  } 1\leqslant i\leqslant p.\label{8.2}\end{equation}

We first show that  $E_i\not \in \hbox{Im}(C_p^T)$ for any given $i$ with  $1\leqslant i\leqslant p$.  Otherwise, there exist an $i$ with  $1\leqslant i\leqslant p$ and a vector $x_i\in\rr^{N-p}$ such that $E_i=C_p^Tx_i$.
Then, it follows from \eqref{8.2} that
$$0 = (E_i, \widehat U)  = (x_i, C_p\widehat U).$$
Since $W=C_p\widehat U$ is the solution to the reduced system \eqref{7.7} with $H = \widehat H$, which  is approximately  controllable, we get  thus $x_i=0$,  which contradicts $E_i\not =0$. Thus, since  $\hbox{dim Im}(C_p^T) = N-p$ and
$\hbox{dim}(V)=p,$ we have $V \oplus \hbox{Im}(C_p^T) = \mathbb R^N$.
Then, for any given  $i$ with $1\leqslant i \leqslant p$, there exists a vector $y_i\in \mathbb R^{N-p}$, such that
$$A^TE_i = \sum_{j=1}^p\alpha_{ij} E_j + C_p^Ty_i.$$
Noting (\ref{8.2})  and applying $E_i$ to system  (\ref{8.1}), it follows that
$$0  = (A\widehat U, E_i) = (\widehat U, A^TE_i) = (\widehat U, C_p^Ty_i) = (C_p\widehat U, y_i). $$
Once again, the approximate controllability of the reduced system \eqref{7.7}   implies that $y_i=0$ for $1\leqslant i \leqslant p$. Then, it follows that
$$A^TE_i = \sum_{j=1}^p\alpha_{ij} E_j,\quad 1\leqslant i\leqslant p.$$
So, the subspace $V$ is invariant for $A^T$.

In \cite{RobinExact}, by the sharp regularity given in  \cite{Lasciecka1, Lasciecka2} on Neumann type mixed problem,     we improved the regularity \eqref{4.3} of the solution to problem \eqref{8.1}.  In fact, setting
\begin{equation} \alpha=\begin{cases}3/5-\epsilon, & \Omega \hbox{ is a bounded smooth domain},\\
3/4-\epsilon,&\Omega \hbox{ is a parallelepiped},\end{cases}\label{8.3}\end{equation}
where $\epsilon>0$ is a sufficiently small number,  the trace
\begin{equation} \widehat U|_{\Gamma_1}\in( H_{loc}^{2\alpha-1}((0, +\infty) \times \Gamma_1))^N\label{8.4}\end{equation}
 with  the corresponding  continuous  dependence with respect to  $ \widehat H$.

Next, noting (\ref{8.2}) and applying $E_i \ (1\leqslant i \leqslant p)$ to  the boundary condition on $\Gamma_1$ in   (\ref{8.1}), we get
\begin{equation}(D^TE_i, \widehat H) = (E_i, B\widehat U).\label{8.4b}\end{equation}
Then, it follows that 
\begin{align}\|(D^TE_i, \widehat H)\|_{H^{2\alpha-1}((0,  T)\times \Gamma_1)}  
\leqslant c\|\widehat U\|_{H^{2\alpha-1}((0,  T)\times \Gamma_1)}.\label{8.4c}\end{align}
On the other hand, by  the continuous dependence \eqref{8.4}, we have
\begin{align}\| \widehat U\|_{H^{2\alpha-1}((0,  T)\times \Gamma_1)}  
\leqslant c\|\widehat H\|_{L^2((0,  T)\times \Gamma_1)}.\label{8.4g}\end{align}
Then inserting \eqref{8.4g} into \eqref{8.4c}, we get 
\begin{align}\|(D^TE_i, \widehat H)\|_{H^{2\alpha-1}((0,  T)\times \Gamma_1)}    
\leqslant c
\|\widehat H\|_{L^2((0, T)\times \Gamma_1)}.\label{8.4e}\end{align}
Taking  $\widehat H=  D^TE_i h$ in \eqref{8.4e}, we get
\begin{align}\|D^TE_i\|\|h\|_{H^{2\alpha-1}((0,  T)\times \Gamma_1)}  
\leqslant c
\|h\|_{L^2((0, +T)\times \Gamma_1)},\quad \forall h\in L^2((0,  T)\times \Gamma_1).\end{align}
Because of the compactness of the embedding $H^{2\alpha-1}((0, T)\times \Gamma_1)$ to $L^2((0,  T)\times \Gamma_1)$ for  $2\alpha-1>0$, we deduce that 
\begin{align}D^TE_i=0, \quad 1\leqslant i\leqslant p.\label{8.4d}\end{align}
Then it follows from \eqref{8.4d} that 
\begin{equation}  V\subseteq \hbox{ Ker}(D^T).\end{equation}
Moreover, for $1\leqslant i \leqslant p$ we have \begin{equation} (E_i, B\widehat U) =0\quad \hbox{on } (0, +\infty)\times \Gamma_1.\label{8.5}\end{equation}

Now, let $x_i\in \mathbb R^{N-p}$, such that
\begin{equation} B^TE_i = \sum_{j=1}^p\beta_{ij} E_j + C_p^Tx_i.\label{8.6}\end{equation}
Noting \eqref{8.2}  and inserting  the expression  \eqref{8.6} into  \eqref{8.5},   it follows  that
$$  (x_i, C_p\widehat U)= 0\quad \hbox{on } (0, +\infty)\times \Gamma_1.$$  Once again, because of the approximate  boundary controllability of the reduced system  \eqref{7.7}, we deduce  that
$x_i=0$ for $1\leqslant i \leqslant p$.  Then, we get
$$B^TE_i = \sum_{j=1}^p\beta_{ij} E_j,\quad 1\leqslant i \leqslant p.$$
So,  the subspace $V$ is also invariant for $B^T$.

Finally, since dim$(V)=p$, by Lemma \ref{th2.1} and Proposition \ref{th7.3},   $\hbox{Ker}(\mathcal R^T) = V$.
Then,  by assertion (i) of Theorem \ref{th7.5},    $ \hbox{Ker}(\mathcal R^T)$ is bi-orthonormal to
Ker$(C_p)$. This achieves the  proof. \end{proof}

Let $d$ be a column vector of  $D$ and be  contained in $\hbox{Ker}(C_p)$. Then it will  be canceled  in the product matrix $C_pD$, therefore it can not give any effect to the reduced system  (\ref{7.7}). However,  the vectors in $\hbox{Ker}(C_p)$  may play an important  role for the approximate boundary  controllability. More precisely, we have the following

\begin{theoreme}   \label{th8.2} Let  $A$ and $B$ satisfy the  conditions of $C_p$-compatibility (\ref{7.3}) and (\ref{7.5}), respectively.
Assume  that  system (\ref{4.1}) is approximately
synchronizable by $p$-groups under the action of a boundary control matrix $D$. Assume furthermore that
\begin{equation}e_1,\cdots, e_p\in\hbox{Im}(D),\label{8.6b}\end{equation}  where  $e_1,\cdots, e_p$ are given by \eqref{7.1}. Then  system (\ref{4.1}) is actually approximately null controllable.
\end{theoreme}
\begin{proof} By Proposition \ref{th4.3}, it is sufficient to show that  the adjoint system (\ref{3.1}) is $ D$-observable. For $1\leqslant r\leqslant  p$, applying $e_r$  to the adjoint system (\ref{3.1}) and noting
$\phi_r =(e_r, \Phi)$, it follows that
\begin{equation}\begin{cases}\phi_r''-\Delta \phi_r +\sum_{s=1}^p\widehat \alpha_{rs}\phi_s=0&\hbox{ in }  (0, +\infty)\times \Omega,\\
\phi_r=0 & \hbox{ on }  (0, +\infty)\times \Gamma_0,\\
\partial_\nu\phi_r +\sum_{s=1}^p\widehat \beta_{rs}\phi_s=0 & \hbox{ on }  (0, +\infty)\times \Gamma_1,\end{cases}\end{equation}
where the constant coefficients  $\widehat \alpha_{rs}$ and $\widehat \beta_{rs}$ are given by
\begin{equation}Ae_r=\sum_{s=1}^p\widehat \alpha_{rs}e_s,\quad Be_r=\sum_{s=1}^p\widehat \beta_{rs}e_s,\quad 1\leqslant r\leqslant p.\end{equation}
On the other hand, noting \eqref{8.6b},  the $ D$-observation \eqref{3.5}
implies that
\begin{equation}\phi_r\equiv 0 \quad
\hbox{ on }  (0, T)\times\Gamma_1\end{equation}
for $1\leqslant r\leqslant p$.
Then,   by Holmgren's uniqueness theorem, we get
\begin{equation}\phi_r\equiv 0\quad \hbox{in } (0,+\infty)\times \Omega \label{8.8}\end{equation}
for $1\leqslant r\leqslant p.$
Thus,   $\Phi\in \hbox{Im}(C_p^T)$, then we can write $\Phi= C_p^T\Psi$ and the adjoint system (\ref{3.1}) becomes
\begin{equation}\begin{cases}C_p^T\Psi''-C_p^T\Delta \Psi+A^TC_p^T\Psi=0&\hbox{ in }  (0, +\infty)\times \Omega,\\
 C_p^T\Psi=0 & \hbox{ on }  (0, +\infty)\times \Gamma_0,\\
 C_p^T\partial_\nu\Psi +  B^TC_p^T\Psi=0 & \hbox{ on }  (0, +\infty)\times \Gamma_1.\end{cases}\end{equation}
Noting the  conditions of $C_p$-compatibility (\ref{7.3}) and (\ref{7.5}), it follows that
\begin{equation}\begin{cases}C_p^T(\Psi''- \Delta \Psi+ \overline A_p^T\Psi)=0&\hbox{ in }  (0, +\infty)\times \Omega,\\
 C_p^T\Psi=0 & \hbox{ on }  (0, +\infty)\times \Gamma_0\\
 C_p^T(\partial_\nu\Psi +   \overline B_p^T\Psi)=0 & \hbox{ on }  (0, +\infty)\times \Gamma_1.\end{cases}\end{equation}
Since the map $C_p^T$ is injective,  we find again the reduced adjoint system \eqref{7.7b}.
Accordingly, the $ D$-observation \eqref{3.5} implies that
\begin{equation}  D^T\Phi \equiv D^TC^T_p\Psi\equiv 0.\end{equation}
Since system (\ref{4.1}) is approximately synchronizable by $p$-groups under the action of the boundary control matrix $D$, by Proposition \ref{th7.a1},   the reduced adjoint system \eqref{7.7b} for $\Psi$ is $C_pD$-observable, therefore, $\Psi\equiv 0$, then  $\Phi\equiv 0$.  So,  the adjoint   system (\ref{3.1}) is $ D$-observable,  then  by Proposition  \ref{th4.3}, system (\ref{4.1})  is approximately  null controllable.
\end{proof}

\vskip 1cm

{\bf Acknowledgement} This work was partially supported by  National Natural Science Foundation of China under Grant 11831011.

\end{document}